\documentclass[oneside,letterpaper,11pt]{amsart}

\usepackage[top=1in,bottom=1in,left=1.55in,right=1.55in]{geometry}
\usepackage[T1]{fontenc}
\usepackage[utf8]{inputenc}

\usepackage{amssymb}
\usepackage{microtype}
\usepackage{graphicx}
\usepackage{enumitem}
\usepackage{mathtools}

\usepackage{comment}

\newtheorem{theorem}{Theorem}[section]
\newtheorem{corollary}[theorem]{Corollary}
\newtheorem{prop}[theorem]{Proposition}
\newtheorem{lemma}[theorem]{Lemma}
\newtheorem*{lemma*}{Lemma}
\newtheorem*{corollary*}{Corollary}

\newcommand{\introthmname}{}
\newtheorem{introthminn}{\introthmname}
\newenvironment{introthm}[1]
  {\renewcommand{\introthmname}{#1}\begin{introthminn}}
  {\end{introthminn}}
  
\newcommand{\adhocrmkname}{}
\newtheorem{adhocrmkinn}[theorem]{\adhocrmkname}
\newenvironment{adhocrmk}[1]
  {\renewcommand{\adhocrmkname}{#1}\begin{adhocrmkinn}}
  {\end{adhocrmkinn}}

\theoremstyle{definition}
\newtheorem{definition}[theorem]{Definition}

\theoremstyle{remark}
\newtheorem{remark}[theorem]{Remark}

\usepackage{accents}

\usepackage{rsfso}
\pdfmapfile{+rsfso.map}
\usepackage{bm}

\newcommand{\RR}{\mathbb{R}} % real numbers
\newcommand{\CC}{\mathbb{C}} % copmlex numbers
\newcommand{\RP}{\mathbb{RP}} % real projective space
\newcommand{\CP}{\mathbb{CP}} % complex projective space
\newcommand{\ZZ}{\mathbb{Z}} % integers
\newcommand{\QQ}{\mathbb{Q}} % rational numbers
\newcommand{\PP}{\mathbb{P}} % projective space

\newcommand{\KL}{\mathbf{L}}

\newcommand{\HL}[1]{H^2(#1;\ZZ)}

\newcommand{\U}{\mathrm{U}}

\newcommand{\br}[1]{\ensuremath{\langle#1\rangle}}
\newcommand{\ct}[2]{\ensuremath{\langle#2\sqcup1\langle#1\rangle\rangle}}

\DeclareMathOperator{\conj}{conj}   % complex conjugation
\DeclareMathOperator{\im}{im}       % image of a map
\DeclareMathOperator{\bl}{bl}       % blow-up
\DeclareMathOperator{\tr}{tr}       % transfer
\DeclareMathOperator{\discr}{discr} % discriminant group
\DeclareMathOperator{\Hom}{Hom}     % Automorphisms of something

\newcommand{\PD}{\mathrm{PD}}                % Poincaré duality
\newcommand{\OH}{\Omega_{h,\sigma}}        % period space of hom. type
\newcommand{\HTp}{{h,\sigma,\phi}}       % homological type
\newcommand{\GHp}{\mathrm{O}(\KL,\HTp)}      % autom. group of homomolgical type
\newcommand{\OHp}{\Omega_\HTp}      % period space of hom. type
\newcommand{\PGL}{\mathrm{PGL}}     % period space of hom. type
\newcommand{\LL}{\mathcal{L}}       % line bundle
\newcommand{\I}{\mathrm{i}}         % sqrt(-1)

\newcommand\os{\hat{s}}
\newcommand{\ncr}{\ensuremath{r}}
\renewcommand{\epsilon}{\varepsilon}
\newcommand{\orth}{\mathbin{\bot} }

\newcommand{\ie}{\emph{i.e.}\ }
\newcommand{\iec}{\emph{i.e.,} }

\usepackage{xcolor}

\renewcommand{\emptyset}{\varnothing}

\usepackage{amsrefs}
\renewcommand{\eprint}[1]{\tt arXiv:\linebreak[0]\href{http://arxiv.org/abs/#1}{\path{#1}}}

\usepackage{tikz}
\usepackage{tikz-cd}
\usetikzlibrary{arrows}

\tikzset{
    vec/.style={
        circle,
        fill,
        inner sep=1pt
    },
    every label/.style={
        rectangle,
        inner sep=0,
        label distance=1.3pt
        }
}

\usepackage{hyperref}
\hypersetup{colorlinks,citecolor=black,filecolor=black,linkcolor=black}

\title{Real nodal sextics without real nodes}
\author{Johannes Josi}

\address{%
Université de Genève
\textup{and} 
Institut de Mathématiques de Jussieu--Paris Rive Gauche}
\curraddr{Section de mathématiques\\
Université de Genève\\
Villa~Battelle\\
Route de Drize~7\\
1227~Carouge, Switzerland}
\email{johannes.josi@unige.ch}
\subjclass[2010]{14P25 (primary), 14H50, 14J28}

\begin{document}
\thispagestyle{empty}

\begin{abstract}
We present a rigid isotopy classification of irreducible sextic curves in $\RP^2$
which have non-real ordinary double points as their only singularities.
Our approach uses periods of K3 surfaces and V.~Nikulin's classification
of ``involutions with condition'' on unimodular lattices.
The classification obtained generalizes Nikulin's rigid isotopy classification
of non-singular sextics in $\RP^2$.
\end{abstract}

\maketitle

\section*{Introduction}\label{sec:intro}

In 1979, V. Nikulin~\cite{Nikulin79} classified real non-singular sextics
(curves of degree six) in $\RP^2$ up to rigid isotopies,
that is, up to deformations in the space of real non-singular sextics.
The aim of this article is to generalize this classification to
real irreducible sextics whose only singularities are non-real
ordinary double points.
For such sextics, rigid isotopy means deformation within the space of real
nodal sextics with a fixed number of nodes.

\subsection*{Statement of the results}\label{sec:intro.results}
We associate the following invariants with a real irreducible sextic $C\subset\CP^2$
whose only singularities are $m$ pairs of non-real nodes:
\begin{itemize}
\item \emph{Isotopy type of the real part.}
Since all the nodes are non-real, the real part
$\RR C\subset \RP^2$ is a collection of smoothly embedded circles, called \emph{ovals}.
The way these ovals are disposed in $\RP^2$ is invariant under rigid isotopies.
\item \emph{Dividing type.}
Recall that a real non-singular curve $C$ is called \emph{dividing} or
\emph{of dividing type~\textup{I}%
\footnote{In the literature this is often simply called the \emph{type}
of a curve, but this might lead to confusion here since we also use other notions
of type, such as the homological type of a curve.}
} if $C\setminus\RR C$ consists of two connected components, and
\emph{of dividing type~\textup{II}} if $C\setminus\RR C$ is connected.
We extend this notion to nodal curves by declaring that a nodal curve $C$
is dividing if its normalization $\tilde C$ is dividing
with respect to the inherited real structure.
If $C$ is a non-singular dividing curve, then the two connected components of
$C\setminus\RR C$ are called \emph{halves} of the curve.
\item \emph{Number of crossing pairs \textup{(only defined for dividing curves)}.}
A non-real node of a dividing curve is called \emph{crossing} if the
two branches intersecting at the node belong to different halves of the normalization of the curve.
A non-real node is crossing if and only if this is the case for its
complex conjugate node, so that we can speak about \emph{crossing}
and \emph{non-crossing pairs} of non-real nodes.
\end{itemize}
These three characteristics are invariant under rigid isotopies.
Our first theorem states that they determine the rigid isotopy class.

\begin{introthm}{Theorem}\label{thm:curve-uniqueness}
Two real irreducible sextics with $m$ pairs of non-real nodes
are rigidly isotopic if and only if their real
parts are isotopic, they are of the same dividing type
and, if they are dividing, they have the same number of crossing pairs.
\end{introthm}

By comparison, Nikulin's theorem states that non-singular real sextics
are classified up to rigid isotopy by their isotopy type and their dividing type.

The second question we consider is which combinations of isotopy type,
dividing type and number of crossing pairs are realized by real
irreducible sextics $C$ with $m$ pairs of non-real nodes.
There are several geometric conditions which must be satisfied.
\begin{enumerate}
\item \emph{Existence of the smoothing.}
Smoothing the non-real nodes leads to a non-singular real sextic
whose real part has the same isotopy type as the original curve.
The smoothed curve is dividing if and only if the original curve is
dividing and without crossing pairs.
Therefore, a sextic with prescribed isotopy type, dividing type and number 
of crossing nodes can only exist if there is a non-singular sextic
with the same isotopy type and the appropriate dividing type.
\item \emph{Harnack's inequality for $\tilde C$.}
The normalization $\tilde C$ is a smooth real curve of genus $g=10-2m$.
Harnack's inequality
(see Harnack~\cite{Harnack} and Klein~\cite{Klein}*{p.~154} for a topological proof)
for $\tilde C$ states that $l \le g+1 = 11-2m$,
where $l$ denotes the number of ovals of $\RR C$.
Moreover, the equality $l = 11 - 2m$ is only possible if $C$ is dividing.
\end{enumerate}

If $C$ is dividing, then there are further restrictions coming from
Rokhlin's complex orientation formula (see~\cite{Rokhlin}).
The orientation of each of the halves of $\tilde C$ induces a boundary orientation
on the real part $\RR C$. A pair of ovals in $\RR C$ is called \emph{injective} if
one oval is contained in the disk bounded by the other oval.
An injective pair is called \emph{positive} if the orientation of the two ovals
is induced by an orientation of the annulus bounded by these ovals, and
\emph{negative} otherwise. Let $\Pi_+$ and $\Pi_-$ denote the numbers of
positive and negative injective pairs, respectively.
In our case, Rokhlin's complex orientation formula states that
\[
2(\Pi_- - \Pi_+) + l = 9 - 2\ncr,
\]
where $l$ is the number of ovals and $\ncr$ is the number of crossing pairs.
This implies the following relations between the
topology of the real part and the number $\ncr$ of crossing pairs for dividing curves:
\begin{enumerate}[resume*]
\item \emph{Arnold's congruence.}
An oval is called \emph{even} if it lies inside an even number of
other ovals, and \emph{odd} otherwise. Arnold's congruence states that
\[o_\mathrm{even}-o_\mathrm{odd} \equiv 9 - 2\ncr \mod{4},\]
where $o_\mathrm{even}$ and $o_\mathrm{odd}$ denote the numbers of even and odd ovals, respectively.
\item\emph{Restriction for curves without injective pairs.}
For curves without injective pairs, the complex orientation
formula implies
\[
l = 9 - 2\ncr.
\]
\end{enumerate} 

Our second theorem states that these conditions are sufficient for
the existence of a real sextic of prescribed isotopy type, dividing type
and number of crossing pairs.

\begin{introthm}{Theorem}\label{thm:curve-existence}
A triple consisting of an isotopy type, a dividing type and a number of crossing pairs (if the dividing type is I)
is realized by a real irreducible sextic whose only singularities are $m$ pairs of non-real nodes
if and only if it satisfies the conditions \textup{(1)--(4)} above.
\end{introthm}

\begin{corollary*}
The rigid isotopy classes of real irreducible nodal sextics without real nodes
are those listed in Figures~\ref{fig:typeI} and \ref{fig:typeII}.
In total, there are $78$ classes of dividing sextics and $125$ classes
of non-dividing sextics.
\end{corollary*}

To describe the isotopy type of the real part, also called real scheme,
we use Viro's notation:
the symbol $\br{n}$ stands for a collection of $n$ empty ovals;
$\br{1\br{A}}$ is obtained from $A$ by adding an oval containing
all of ovals $A$ in its interior; if $A$ and $B$ are collections of ovals,
then $\br{A\sqcup B}$ denotes the  disjoint union of $A$ and $B$,
such that no oval of $A$ lies in the interior of an oval of $B$ and vice versa.

\subsection*{Organization of the paper}
In Section~\ref{sec:periods}, we review the construction of the K3~surface obtained as a
double covering of $\CP^2$ ramified along a sextic curve $C\subset\CP^2$,
first in the complex case and then in the real case.
This allows us to associate certain homological data, called
the \emph{homological type}, with each real nodal sextic.
We introduce the period map for (suitably marked, polarized) real K3 surfaces
and prove that rigid isotopy classes of real sextics of a fixed homological type
are in bijection with connected components of the corresponding period space,
modulo the action of the orthogonal group of the homological type (Proposition~\ref{prop:period_ri}).

In Section~\ref{sec:quasisimple} we show that, for real nodal irreducible
sextics without real nodes, the aforementioned orthogonal group acts
transitively on the connected components of the corresponding period space (Proposition~\ref{prop:reflections}).
Therefore, such sextics are rigidly isotopic if and only if they have the same
homological type.

In Section~\ref{sec:topol} we establish a correspondence between certain
topological properties of sextics on one hand and arithmetical properties
of their homological types on the other hand.
In particular, we show how one can tell from the homological type whether
a sextic is dividing or not,
and we study how the homological type changes when a pair of non-real nodes is perturbed.

Section~\ref{sec:classif}, which is the most technical part of the paper,
is devoted to the classification of the homological types of real irreducible sextics without real nodes.
We define a set of numerical invariants which completely determines these homological types (Theorem~\ref{thm:unique}),
and we give necessary and sufficient conditions for the existence of a homological type
in terms of these invariants  (Theorem~\ref{thm:conditions}).
The proof of these two theorems is based on Nikulin's classification of
``involutions with condition''~\cite{Nikulin83}.

Finally, in Section~\ref{sec:thm12}, we establish a complete dictionary between
the topological invariants mentioned in the introduction and the arithmetical invariants
introduced in Section~\ref{sec:classif}.
This allows us to prove Theorems~\ref{thm:curve-uniqueness} and \ref{thm:curve-existence}.

\begin{figure}[hb]
\makebox[\textwidth][c]{
\begin{tikzpicture}[x=.7cm,y=.55cm]
\tikzset{
    every label/.style={
        rectangle,
        inner sep=0,
        label distance=1.3pt
        }
}
\Small
\node at (-1,12) () {$m_{\text{max}}$};

\node at (-1,11) () {$0$};
\node[label=south:{$\{0\}$}] at ( 1,11) () {\ct{9}{1}};
\node[label=south:{$\{0\}$}] at ( 9,11) () {\ct{5}{5}};
\node[label=south:{$\{0\}$}] at (17,11) () {\ct{1}{9}};

\node at (-1,9) () {$1$};
\node[label=south:{$\{0\}$}] at ( 1,9) () {\br{1\br{8}}};
\node[label=south:{$\{1\}$}] at ( 3,9) () {\ct{7}{1}};
\node[label=south:{$\{0\}$}] at ( 5,9) () {\ct{6}{2}};
\node[label=south:{$\{1\}$}] at ( 7,9) () {\ct{5}{3}};
\node[label=south:{$\{0\}$}] at ( 9,9) () {\ct{4}{4}};
\node[label=south:{$\{1\}$}] at (11,9) () {\ct{3}{5}};
\node[label=south:{$\{0\}$}] at (13,9) () {\ct{2}{6}};
\node[label=south:{$\{1\}$}] at (15,9) () {\ct{1}{7}};
\node[label=south:{$\{0\}$}] at (17,9) () {\br{9}};

\node at (-1,7) () {$2$};
\node[label=south:{$\{1  \}$}] at ( 3,7) () {\br{1\br{6}}};
\node[label=south:{$\{0,2\}$}] at ( 5,7) () {\ct{5}{1}};
\node[label=south:{$\{1  \}$}] at ( 7,7) () {\ct{4}{2}};
\node[label=south:{$\{0,2\}$}] at ( 9,7) () {\ct{3}{3}};
\node[label=south:{$\{1  \}$}] at (11,7) () {\ct{2}{4}};
\node[label=south:{$\{0,2\}$}] at (13,7) () {\ct{1}{5}};
\node[label=south:{$\{1  \}$}] at (15,7) () {\br{7}};

\node at (-1,5) () {$3$};
\node[label=south:{$\{0,2\}$}] at ( 5,5) () {\br{1\br{4}}};
\node[label=south:{$\{1,3\}$}] at ( 7,5) () {\ct{3}{1}};
\node[label=south:{$\{0,2\}$}] at ( 9,5) () {\ct{2}{2}};
\node[label=south:{$\{1,3\}$}] at (11,5) () {\ct{1}{3}};
\node[label=south:{$\{2  \}$}] at (13,5) () {\br{5}};

\node at (-1,3) () {$4$};
\node[label=south:{$\{1,3\}$}] at ( 7,3) () {\br{1\br{2}}};
\node[label=south:{$\{2,4\}$}] at ( 9,3) () {\ct{1}{1}};
\node[label=south:{$\{3  \}$}] at (11,3) () {\br{3}};

\node[label=south:{$\{0  \}$}] at (16,3) () {\br{1\br{1\br{1}}}};

\node at (-1,1) () {$5$};
\node[label=south:{$\{4  \}$}] at ( 9,1) () {\br{1}};
\end{tikzpicture}
}
\caption{Rigid isotopy classes of dividing real nodal sextics
with $m$ pairs of non-real nodes of which $\ncr$ are crossing.
The set below each real scheme indicates the possible values for the number of crossing pairs $\ncr$.
The total number of pairs $m$ can take any value between $\ncr$ and the upper bound $m_\text{max}$
indicated on the left of each row.}
\label{fig:typeI}
\end{figure}
\begin{figure}[hb]
\makebox[\textwidth][c]{
\begin{tikzpicture}[x=.7cm,y=.55cm]
\Small
\node at (-2,11) () {$m_{\text{max}}$};

\node at (-2,10) () {$0$};
\node at ( 0,10) () {\br{1\br{9}}};
\node at ( 2,10) () {\ct{8}{1}};
\node at ( 8,10) () {\ct{5}{4}};
\node at (10,10) () {\ct{4}{5}};
\node at (16,10) () {\ct{1}{8}};
\node at (18,10) () {\br{10}};

\node at (-2,9) () {$0$};
\node at ( 1,9) () {\br{1\br{8}}};
\node at ( 3,9) () {\ct{7}{1}};
\node at ( 7,9) () {\ct{5}{3}};
\node at ( 9,9) () {\ct{4}{4}};
\node at (11,9) () {\ct{3}{5}};
\node at (15,9) () {\ct{1}{7}};
\node at (17,9) () {\br{9}};

\node at (-2,8) () {$1$};
\node at ( 2,8) () {\br{1\br{7}}};
\node at ( 4,8) () {\ct{6}{1}};
\node at ( 6,8) () {\ct{5}{2}};
\node at ( 8,8) () {\ct{4}{3}};
\node at (10,8) () {\ct{3}{4}};
\node at (12,8) () {\ct{2}{5}};
\node at (14,8) () {\ct{1}{6}};
\node at (16,8) () {\br{8}};

\node at (-2,7) () {$1$};
\node at ( 3,7) () {\br{1\br{6}}};
\node at ( 5,7) () {\ct{5}{1}};
\node at ( 7,7) () {\ct{4}{2}};
\node at ( 9,7) () {\ct{3}{3}};
\node at (11,7) () {\ct{2}{4}};
\node at (13,7) () {\ct{1}{5}};
\node at (15,7) () {\br{7}};

\node at (-2,6) () {$2$};
\node at ( 4,6) () {\br{1\br{5}}};
\node at ( 6,6) () {\ct{4}{1}};
\node at ( 8,6) () {\ct{3}{2}};
\node at (10,6) () {\ct{2}{3}};
\node at (12,6) () {\ct{1}{4}};
\node at (14,6) () {\br{6}};

\node at (-2,5) () {$2$};
\node at ( 5,5) () {\br{1\br{4}}};
\node at ( 7,5) () {\ct{3}{1}};
\node at ( 9,5) () {\ct{2}{2}};
\node at (11,5) () {\ct{1}{3}};
\node at (13,5) () {\br{5}};

\node at (-2,4) () {$3$};
\node at ( 6,4) () {\br{1\br{3}}};
\node at ( 8,4) () {\ct{2}{1}};
\node at (10,4) () {\ct{1}{2}};
\node at (12,4) () {\br{4}};

\node at (-2,3) () {$3$};
\node at ( 7,3) () {\br{1\br{2}}};
\node at ( 9,3) () {\ct{1}{1}};
\node at (11,3) () {\br{3}};

\node at (-2,2) () {$4$};
\node at ( 8,2) () {\br{1\br{1}}};
\node at (10,2) () {\br{2}};

\node at (-2,1) () {$4$};
\node at ( 9,1) () {\br{1}};

\node at (-2,0) () {$5$};
\node at ( 8,0) () {$\emptyset$};
\end{tikzpicture}
}
\caption{Rigid isotopy classes of non-dividing real nodal sextics
with $m$ pairs of non-real nodes.
The number of pairs of non-real nodes $m$ can take any value between 0
and the upper bound $m_\text{max}$ indicated on the left of each row.}
\label{fig:typeII}
\end{figure}

\section{Nodal sextics, K3 surfaces and periods}\label{sec:periods}

In this section we review some facts about K3 surfaces and their periods,
and in particular the bijection between rigid isotopy classes of
real sextics and chambers of the corresponding real period space,
up to the action of a discrete group.
For details about K3 surfaces and their periods in general, we refer to the book
\cite{BPV} by W.~Barth, K.~Hulek, C.~Peters and A.~van der Ven.
For details about ``weakly polarized'' K3 surfaces,
see D.~Morrison's article \cite{MorrisonRemarks}.
The case of K3 surfaces obtained from complex nodal sextics is also treated in
D.~Morrison and M.~Saito's article \cite{MorrisonSaito}.
We borrow most of the notation from A.~Degtyarev~\cite{Degt}.
The passage from the complex to the real case is done in analogy with
similar situations treated by A.~Degtyarev, I.~Itenberg, V.~Kharlamov and V.~Nikulin \cite{Nikulin79,Itenberg,DIK,Kharlamov}.

\subsection{Complex nodal sextics and their periods}\label{sec:periods.complex}
\subsubsection*{The weakly polarized K3 surface obtained from a nodal sextic}
A \emph{K3 surface} is a non-singular, compact complex surface $X$
which is simply connected and has trivial canonical bundle.
Endowed with the intersection pairing, which we denote by $x\cdot y$,
the cohomology group $\HL{X}$ is an even unimodular lattice of signature $(3,19)$.
Note that all even unimodular lattices of signature $(3,19)$ are isometric.

A \emph{weak polarization} of a K3 surface $X$ is a class $h\in H^2(X,\ZZ)$
of a big and nef line bundle on $X$, that is, a class $h$ such that 
$h^2 > 0$ and $h\cdot D \ge 0$ for all irreducible curves $D\subset X$.
A \emph{weakly polarized K3 surface} is a pair $(X,h)$ formed by a K3 surface $X$
and a weak polarization $h$ of $X$.
The \emph{nodal classes} of a weakly polarized K3 surface $(X,h)$ are the
classes $s\in\HL{X}$ of smooth rational curves on $X$ for which $s\cdot h = 0$.
Nodal classes are \emph{roots}, \ie vectors with self-intersection  $-2$.

Let $C\subset\CP^2$ be a reduced (not necessarily irreducible) nodal sextic.
Let $\bl:V\to\CP^2$ be the blow-up of $\CP^2$ in the nodes of $C$,
and let $\tilde C\subset V$ be the strict transform of $C$.
Let $\eta:X\to V$ be the double covering of $V$ branched along $\tilde C$, and
denote the composition $\bl\circ\,\eta$ by $\pi$.
Alternatively, $\pi$ can be obtained by first taking the (singular) double covering
$Y\to\CP^2$ branched along $C$ and then taking the minimal resolution
$X\to Y$,
\ie we have the following commutative diagram:
\begin{center}
\begin{tikzcd}
        &[-30pt]  X\arrow{r}{} \arrow{d}{\eta} \arrow{dr}{\pi} & Y \arrow{d}{} &[-30pt] \\
\tilde C \subset & V\arrow{r}{\bl} & \CP^2& \supset C
\end{tikzcd}
\end{center}

The surface $X$ is a K3 surface.
Let $h \in \HL{X}$ be the first Chern class of the line bundle $\pi^\ast(\mathcal{O}_{\CP^2}(1))$.
Then $(X,h)$ is a weakly polarized K3 surface with $h^2 = 2$, which we call
the \emph{weakly polarized K3 surface obtained from the sextic $C$}.
Let $\sigma\subset \HL{X}$ be its set of nodal classes.
The nodal curves are precisely the fibres of $\pi$ above the nodes of $C$.
Let $\Sigma\subset \HL{X}$ be the sublattice generated by $\sigma$, and
let $S = \Sigma\oplus\ZZ h\subset \HL{X}$.

\begin{prop}[{Morrison and Saito \cite{MorrisonSaito}*{Lemma~4.1}}]\label{prop:irred}
The curve $C$ is irreducible if and only if $S$ is a primitive sublattice of $\HL{X}$.
\end{prop}

\begin{remark}
More generally, if $\tilde S$ denotes the primitive closure of $S$ in $\HL{X}$,
then the quotient $\tilde S/S$ is isomorphic to $(\ZZ/2\ZZ)^{k-1}$
where $k$ is the number of irreducible components of $C$.
\end{remark}

The following proposition follows from A.~Mayer's results on linear systems
on K3 surfaces \cite{Mayer}, see also \cite{MorrisonSaito}*{Lemma~4.2} and
\cite{Degt}*{Proposition~3.3.1}.

\begin{prop}\label{prop:mayer}
Let $(X,h)$ be a weakly polarized K3 surface with $h^2=2$, and let $\sigma\subset\HL{X}$
be its set of nodal classes.
Then the linear system $|h|$ determines a map $X\to\CP^2$ which is
the minimal resolution of a double covering branched along a nodal sextic $C\subset\CP^2$
if and only if the following conditions hold.
\begin{enumerate}[label=\textup{(\textit{\roman*})},noitemsep]
\item\label{it1} The nodal classes are pairwise orthogonal.
\item\label{it2} If $s\in L$ is a root such that $2s\in\Sigma$, then $s\in\Sigma$.
\item\label{it3} There is no root $s\in\sigma$ such that $s+h\in 2L$.
\end{enumerate}
\end{prop}

\begin{remark}
If the sublattice $S\subset \HL{X}$ generated by $h$ and $\sigma$ is primitive,
then conditions \ref{it2} and \ref{it3} are satisfied.
\end{remark}

\subsubsection*{Periods}

Fix an integer $n\le 10$, interpreted as the number of nodes of the sextics under consideration.
We fix once and for all an even unimodular ``model lattice'' $\KL$ of signature (3,19).
Consider pairs $(h,\sigma)$, where
$h\in \KL$ is a vector of square $2$ and $\sigma\subset\KL$ is a set of $n$
pairwise orthogonal roots orthogonal to $h$,
such that the sublattice $S\subset\KL$ generated by $\sigma$ and $h$ is primitive.
We regard two pairs, say $(h,\sigma)$ and $(h',\sigma')$, as equivalent
if there is an auto-isometry of $\KL$ taking $h$ to $h'$ and $\sigma$ to $\sigma'$.

\begin{prop}[\cite{MorrisonSaito}*{Theorem~3.2}]\label{prop:uniquetriple}
All pairs $(h,\sigma)$ as above are equivalent.
\end{prop}

In view of the above proposition, we fix a pair $(h,\sigma)$ where
$h\in\KL$ is a vector of square 2 and $\sigma\subset\KL$ is a set of
$n$ pairwise orthogonal roots, orthogonal to $h$,
such that the sublattice $S\subset\KL$ generated by $\sigma$ and $h$ is primitive.
Let $\Sigma\subset\KL$ denote the sublattice generated by $\sigma$.

\begin{definition}
An \emph{$(h,\sigma)$-marking} of a K3 surface $X$ is an isometry
$\mu:\KL\to\HL{X}$ such that $\mu(h)$ is a weak polarization of $X$
and the set of nodal classes of $(X,\mu(h))$ is $\mu(\sigma)$.
An \emph{$(h,\sigma)$-marked K3 surface} is a pair $(X,\mu)$ formed by
a K3 surface $X$ and an $(h,\sigma)$-marking $\mu$ of $X$.
\end{definition}

Note that a weakly polarized K3 surface admits an $(h,\sigma)$-marking
if and only if it can be obtained from an irreducible sextic $C\subset\CP^2$ with $n$ nodes.

The Hodge structure of a K3 surface $X$ is determined by the position of
the complex line $H^{2,0}(X)\subset H^2(X;\CC)$.
Every non-zero class $\omega_X\in H^{2,0}(X)$
is represented by a non-vanishing holomorphic 2-from on $X$.
It satisfies $\omega_X^2 = 0$ and $\omega_X\cdot\bar\omega_X > 0$.
Let us denote by $\KL_\CC$ the complex vector space $\KL\otimes_\ZZ\CC$.
The integral bilinear form defined on $\KL$ naturally extends to a
complex-valued bilinear form on $\KL_\CC$.

\begin{definition}
The \emph{period space $\Omega_{h,\sigma}^0$ for $(h,\sigma)$-marked K3 surfaces}
is defined as follows:
\begin{align*}
\Omega_{h,\sigma} & =\mathrlap{\left.\{\omega \in \KL_\CC \mid \omega^2 = 0,\;\omega \cdot \bar\omega > 0,\;
\omega\cdot h=0,\;\omega\orth\sigma \} \right/ \CC^\ast \subset \PP(\KL_\CC)} && \\
\Omega_{h,\sigma}^0 & = \,\OH \setminus \textstyle\bigcup_{\delta\in\Delta} H_\delta, \quad\text{ where } &
H_\delta &= \{[\omega]\in\OH \mid \omega\cdot \delta = 0\},\\
&& \Delta &= \{\delta\in\KL\mid \delta^2=-2,\delta\cdot h=0, \delta\not\in\Sigma\}.
\end{align*}
The \emph{period} of an $(h,\sigma)$-marked K3 surface $(X,\mu)$
is defined as $[\mu^{-1}_\ast(\omega_X)] \in \PP(\KL_\CC)$ where
$\omega_X$ is a non-zero class in $H^{2,0}(X)$ and $\mu_\ast$
denotes the extension of $\mu$ to a map $\KL_\CC\to H^2(X,\CC)$.
\end{definition}

\begin{remark}\ 
\begin{enumerate}
\item The period space $\OH^0$ is an open set (in the classical topology)
in a quadric of dimension $19-n$.
In particular, it is a non-compact complex manifold.
\item The periods of $(h,\sigma)$-marked K3 surfaces lie in $\OH^0\subset\PP(\KL_\CC)$.
\end{enumerate}
\end{remark}

By a \emph{family} of complex manifolds we mean a proper surjective submersion
$f:\mathcal{X}\to B$ with connected fibres,
where $\mathcal{X}$ and $B$ are connected smooth manifolds,
together with a complex structure defined on each fibre of $f$, varying smoothly along the base.
Given some type of complex manifolds (such as curves, K3 surfaces, etc.),
by a family of this type we mean a family whose fibres are of the said type.
 
\begin{definition}
A \emph{family of $(h,\sigma)$-marked K3 surfaces} is a pair $(\mathcal{X},\mu)$,
where $f:\mathcal{X}\to B$ is a family of K3 surfaces
and $\mu: \underline\KL_B\to R^2f_\ast(\ZZ)$ is a sheaf isomorphism such that
$\mu_b : \KL\to\HL{X_b}$ is an $(h,\sigma)$-marking of $X_b$ for each $b\in B$.
\end{definition}

The next theorem is a special case of an analogous statement
for weakly polarized marked K3 surfaces by Morrison~\cite{MorrisonRemarks}.
It is based on the Global Torelli Theorem for K3 surfaces by I.~Shafarevich and I.~Piatetski-Shapiro~\cite{Shafarevich}
and on the surjectivity of the period map by V.~Kulikov~\cite{Kulikov}.
See also \cite{Degt}*{Theorem~3.4.1} and \cite{MorrisonSaito}*{Theorem~4.3}.

\begin{theorem}[Morrison~\cite{MorrisonRemarks}]\label{thm:moduli}
The space $\OH^0$ is a fine moduli space for $(h,\sigma)$-marked K3 surfaces.
\end{theorem}

This means that isomorphism classes of families of $(h,\sigma)$-marked K3 surfaces over $B$
are naturally in bijection with smooth maps from $B$ to $\OH^0$.
In particular, there exists a universal family of $(h,\sigma)$-marked K3 surfaces over $\OH^0$
from which every other family is obtained as a pull-back.

\begin{remark}\label{rk:mayer2}
The equivalence between $(h,\sigma)$-marked K3 surfaces and nodal sextics
extends to families of $(h,\sigma)$-marked K3 surfaces.
More precisely, a family of $(h,\sigma)$-marked K3 surfaces  $(\mathcal{X},\mu)$
determines a bundle of projective planes $\PP(E)\to B$ and an equisingular
family of nodal sextics $\mathcal{C}\subset \PP(E)$ over $B$.
Indeed, if $(\mathcal{X},\mu)$ is a family of $(h,\sigma)$-marked K3 surfaces,
then $\mu_b^{-1}(h)$ is a class of Hodge type $(1,1)$ on each fibre $X_b$.
By the variational Lefschetz theorem on $(1,1)$-classes, there is a line bundle
$\LL$ on $\mathcal{X}$ such that $\mu_b^{-1}(h)$ is the class of the
restriction $\LL|_b$ for each $b\in B$.
Since $\dim\Gamma(X_b,\LL|_b)=3$ independently of $b\in B$,
the direct image sheaf $\pi_\ast\LL$ is locally free and defines
a rank 3 vector bundle $E\to B$.
The line bundle $\LL$ defines a morphism $\mathcal{X} \to \PP(E)$ which is
the minimal resolution of a double covering branched along an equisingular
family of nodal sextics $\mathcal{C}\subset\PP(E)$ over $B$.
\end{remark}

\subsection{Real nodal sextics and their periods}\label{sec:periods.real}

\subsubsection*{Real structures}
A \emph{real structure} on a complex (algebraic or analytic) variety $X$
is an anti-holomorphic involution $c:X\to X$.
Equivalently, a real structure can be viewed as an isomorphism between
$X$ and $\bar{X}$, where $\bar{X}$ is the complex variety with the same underlying
space as $X$, but opposite complex structure, \iec
$\mathcal{O}_{\bar{X}}(U) := \{ \bar{f} \mid f\in\mathcal{O}_X(U)\}$.
A \emph{real variety} is a pair $(X,c)$ formed by a complex variety $X$ and
a real structure $c$ defined on it.
The \emph{real points} of a real variety are those points which are fixed by
the real structure.
By a real structure on a family of complex manifolds, say $f:\mathcal{X}\to B$,
we mean a smooth involution $c$ on $\mathcal{X}$ such that $f\circ c = f$
and the restriction of $c$ to each fibre is anti-holomorphic.

\subsubsection*{Real sextics and real K3 surfaces}
Suppose that $C\subset\CP^2$ is a real sextic, \ie it is invariant under complex conjugation,
or equivalently, defined by a polynomial with real coefficients.
Let $(X,h)$ be the weakly polarized K3 surface obtained from $C$.
As in \S~\ref{sec:periods.complex}, let $\pi:X\to\CP^2$ be the composition
$\bl\circ\,\eta$, where $\bl:V\to \CP^2$ denotes the blow-up of $\CP^2$ in the nodes of $C$ and
$\eta:X\to V$ is the double covering of $V$ branched along the strict transform of $C$.
There are two real structures $c_1, c_2$ on $X$
such that  $\pi\circ c_i=\conj\,\circ\,\pi$,
where $\conj:\CP^2\to\CP^2$ denotes the complex conjugation.
We have $c_2 = \tau\circ c_1 = c_1\circ\tau$,
where $\tau:X\to X$ is the deck transformation of $\eta$.

Let $c$ be one of the two real structures.
It induces an involutive isometry $c^\ast$ on the lattice $\HL{X}$.
Note that $c$ maps algebraic curves on $X$ to algebraic curves, but
it inverts the orientation.
We have $c^\ast(h) = -h$ and $c^\ast(\sigma) = -\sigma$,
where $\sigma$ is the set of nodal classes of $(X,h)$.
If $s\in\sigma$ is the class of an exceptional curve over a real node,
then we have $c^\ast(s) = -s$.
If $s'$ and $s''$ are the classes of the exceptional curves over a pair of non-real nodes,
then we have $c^\ast(s')=-s''$ and $c^\ast(s'') = -s'$.
Moreover, $c^\ast$ is an anti-Hodge isometry on $H^2(X;\CC)$,
\ie it maps forms of type $(p,q)$ to forms of type $(q,p)$.
Therefore, the map $x\mapsto c^\ast(\bar x)$ defines a real structure on the complex line $H^{2,0}(X)$.
Let $\omega_X\in H^{2,0}$ be a nonzero class such that $c^\ast(\bar\omega_X)=\omega_X$.
If we write $\omega_X=\omega_+ + \I\,\omega_-$ with $\omega_\pm\in H^2(X,\RR)$,
then the conditions $\omega_X^2=0$ and $\omega_X\cdot\bar\omega_X>0$
imply that $\omega_+^2 = \omega_-^2 > 0$ and $c^\ast(\omega_\pm)=\pm\,\omega_\pm$.
Therefore the three vectors $h,\omega_+$ and $\omega_-$ are pairwise
orthogonal and span a three-dimensional positive definite subspace of $H^2(X,\RR)$.
In particular, the $c^\ast$-invariant sublattice of $\HL{X}$
has one positive square.

\subsubsection*{Choice of the real structure}
It is not possible to coherently pick one of the two real structures
for all real nodal sextics.
This is possible however if we consider only real nodal sextics without
real nodes. Let $c_1$ and $c_2$ be the two real structures, and
let $X^{c_1}, X^{c_2}\subset X$ be their respective sets of fixed points.
We have $X^{c_i} = \pi^{-1}(B_i)$, where $B_1$ and $B_2$ are the two
(not necessarily connected) regions delimited by $\RR C$ inside $\RP^2$;
more precisely, if $f\in\RR[x_0,x_1,x_2]$ is a polynomial defining $C$,
then we have $B_1 = \{x\in\RP^2 \mid f(x)\ge 0\}$ and
$B_2 = \{x\in\RP^2 \mid f(x)\le 0\}$ or vice versa.
If $C$ is a sextic with only non-real nodes, then,
like in the case of non-singular sextics,
exactly one of the two regions $B_1$ and $B_2$ is non-orientable.
The real structure $c$ whose fixed point set covers the non-orientable
region of $\RP^2\setminus\RR C$
is distinguished by the following homological property
(see~\cite{Nikulin79}*{p.~163}):
\begin{equation}\tag{$\ast$}\label{stern}
\text{There is an element } x\in \HL{X} \text{ with }
c^\ast(x)=-x \text{ and } x\cdot h\equiv 1\pmod{2}.
\end{equation}
In the special case where the real part of the curve $C$ is empty,
one real structure has no fixed points, and the fixed point set of the other
real structure is a non-trivial unramified covering of $\RP^2$.
In this case the real structure with property~(\ref{stern}) is the one
without fixed points.
For the classification, it is convenient to consistently consider one
of the two real structures.
In the following, we always suppose that the real structure $c$ with
property~(\ref{stern}) is chosen.
We call the triple $(X,h,c)$ the \emph{weakly polarized real K3 surface
obtained from $C$}.

\subsubsection*{Real periods}

As in \S~\ref{sec:periods.complex}, we fix a vector $h\in\KL$ and a set of nodal classes $\sigma\subset\KL$.

\begin{definition}
An involutive isometry $\phi$ on $\KL$ is called \emph{geometric}
if $\phi(h) = -h$, $\phi(\sigma)=-\sigma$ and
the $\phi$-invariant sublattice $\KL_+ = \{x\in\KL\mid \phi(x)=x\}$
has one positive square.
\end{definition}
\begin{definition}
Fix a geometric involution $\phi$ on $\KL$.
An \emph{$(h,\sigma,\phi)$-marking} of a real K3 surface $(X, c)$ is an
$(h,\sigma)$-marking of $X$ such that $\mu\circ\phi = c^\ast\circ\mu$.
An \emph{$(h,\sigma,\phi)$-marked K3 surface} is a triple $(X,c,\mu)$
where $(X,c)$ is a real K3 surface and $\mu$ is an $(h,\sigma,\phi)$-marking of $(X,c)$.
A real sextic $C$ is \emph{of homological type $(\KL,h,\sigma,\phi)$} if
the real weakly polarized K3 surface obtained from $C$ admits a $(h,\sigma,\phi)$-marking.
\end{definition}

\begin{definition}
The \emph{period space $\OHp^0$ for $(h,\sigma,\phi)$-marked real K3 surfaces}
is defined as follows:
\begin{align*}
\OHp & =\mathrlap{\left.\{\omega \in \KL_\CC \mid \omega^2 = 0,\;\omega \cdot \bar\omega > 0,\;
\omega\cdot h=0,\;\omega\orth\sigma, \phi_\ast(\omega)=\bar\omega \} \right/ \RR^\ast} && \\
\OHp^0 & = \,\OHp \setminus \textstyle\bigcup_{\delta\in\Delta} H_\delta, \quad\text{ where } &
H_\delta &= \{[\omega]\in\OHp \mid \omega\cdot \delta = 0\},\\
&& \Delta &= \{\delta\in\KL\mid \delta^2=-2,\delta\cdot h=0, \delta\not\in\Sigma\}.
\end{align*}
The \emph{period} of an $(h,\sigma,\phi)$-marked real K3 surface $(X,c,\mu)$
is defined as $[\mu^{-1}_\ast(\omega_X)] \in \OHp^0$, where
$\omega_X$ is a non-zero class in $H^{2,0}(X)$ such that $c^\ast\omega_X=\bar\omega_X$,
and $\mu_\ast$ denotes the extension of $\mu$ to a map $\KL_\CC\to H^2(X,\CC)$.
\end{definition}

The fact that $\phi$ is geometric implies that $\OHp$ is non-empty.
The subspaces $H_\delta\subset\OHp$ can be of real codimension one or two
depending on $\delta$. Therefore the real period space $\OHp^0$ is non-empty,
but in general not connected.

\begin{definition}
A \emph{family of $(h,\sigma,\phi)$-marked real K3 surfaces} is a triple $(\mathcal{X},\mu,c)$,
where $(\mathcal{X},\mu)$ is a family of $(h,\sigma)$-marked K3 surfaces
$c$ is and a real structure on $\mathcal{X}$, such that
$\mu_b : \KL\to\HL{X_b}$ is an $(h,\sigma,\phi)$-marking of $(X_b,c_b)$ for each $b\in B$,
where $c_b$ denotes the restriction of $c$ to $X_b$.
\end{definition}

If $\phi$ is a geometric involution, then $x\mapsto\phi(\bar x)$
defines a real structure on the complex period space $\OH^0$, and
the real period space $\OHp^0\subset\OH^0$ is the set of real points
for this real structure.

\begin{lemma}
Let $(\mathcal{X},\mu)$ be a family of $(h,\sigma)$-marked sextics
whose periods are contained in $\OHp^0\subset\OH^0$.
Then there exists a unique real structure $c$ on $\mathcal{X}$
such that $(\mathcal{X},\mu,c)$ is a family of $(h,\sigma,\phi)$-marked real K3 surfaces.
\end{lemma}

\begin{proof}
Consider the family $\overline{\mathcal{X}}$ obtained from $\mathcal{X}$
by inverting the complex structure on each fibre.
Then $\mu\circ\phi$ defines an $(h,\sigma)$-marking of $\overline{\mathcal{X}}$,
and the periods of the marked families $(\mathcal{X},\mu)$ and
$(\overline{\mathcal{X}},\mu\circ\phi)$ coincide.
Since $\OH^0$ is a fine moduli space for $(h,\sigma)$-marked sextics,
this implies that there is a unique isomorphism $c:\mathcal{X}\to\overline{\mathcal{X}}$
such that $\mu\circ\phi = c^\ast\circ\mu$.
In other words, $c$ defines a real structure on $\mathcal{X}$ such that
$(\mathcal{X},\mu,c)$ is a family of $(h,\sigma,\phi)$-marked real K3 surfaces.
\end{proof}

\begin{corollary}
The space $\OHp^0$ is a fine moduli space for $(h,\sigma,\phi)$-marked real K3 surfaces.
\end{corollary}

\noindent The \emph{orthogonal group} of a homological type $(\KL,\HTp)$ is
\[
\GHp = \{ \psi:\KL\to\KL \text { isometry } \mid \psi(h)=h, \psi(\sigma)=\sigma
\text{ and } \psi\circ\phi=\phi\circ\psi \}.
\]

\begin{prop}\label{prop:period_ri}
The period map induces a one-to-one correspondence between
rigid isotopy classes of real irreducible nodal sextics of
homological type $(\KL,h,\sigma,\phi)$ and
connected components of $\OHp^0$ modulo the action of $\GHp$.
\end{prop}

\begin{proof}
Consider the space $\RP H^0(\CP^2,\mathcal{O}_{\CP^2}(6))^\ast \cong\RP^{27}$
parameterizing real plane sextics, and let
$W\subset\RP^{27}$ be the subset parameterizing
real irreducible nodal sextics $C\subset\CP^2$
of homological type $(\KL,h,\sigma,\phi)$.
Rigid isotopy classes of such sextics are connected components of $W$.
Let $U$ be the set of pairs $(C,\mu)$ where $C\in W$ and $\mu$ is an
$(h,\sigma,\phi)$-marking of the weakly polarized real K3 surface obtained from $C$.
Let $p_1:U\to W$ be the map forgetting the marking.

Consider the universal curve $\mathcal{C}_W = \{(C,x)\in W\times\CP^2 \mid x\in C\}$ over $W$, 
and let $\mathcal{X}_W$ be the family of K3 surfaces obtained from $\mathcal{C}_W\subset W\times\CP^2$.
The projection $\mathcal{X}_W\to W$ is topologically a locally trivial fibration.
Therefore, once we fix an open contractible subset $N \subset W$,
we have natural isometries between $\HL{X_C}$ and $\HL{X_{C'}}$ for $C,C'\in N$,
\ie  markings can be extended to nearby sextics.
This defines a topology on $U$ such that the projection $p_1: U\to W$
is a regular, unramified covering with deck transformation group $\GHp$.
Therefore, $p_1$ induces a bijection between $\pi_0(U)/\GHp$ and
$\pi_0(W)$.

Let $p_2:U\to\OHp^0$ be the period map.
The group $\PGL_3(\RR)$ naturally acts on $U$, and the period map
is invariant with respect to this action.
We claim that $p_2:U\to\OHp^0$ is a $\PGL_3(\RR)$-principal bundle.
Assuming this claim, $p_2$ is a locally trivial fibration with connected fibres,
so it induces a bijection between $\pi_0(U)$ and $\pi_0(\OHp^0)$.
Moreover, since $p_2$ is equivariant with respect to the action of
$\GHp$,
it gives rise to a bijection between the quotients
$\pi_0(U)/\GHp$
and $\pi_0(\OHp^0)/\GHp$,
which proves the proposition.

To prove the claim, consider the universal
family $(\mathcal{X},\mu,c)$ of real $(\HTp)$-marked K3 surfaces over
the real period space $\OHp^0$.
Let $\LL\to\mathcal{X}$ be the line bundle defined by the polarization
and let $E$ be the rank 3 vector bundle over $\OHp^0$ defined
by $\pi_\ast\LL$ (cf. Remark~\ref{rk:mayer2}).
The real structure $c$ induces a real structure on $E$;
let $\RR E$ denote the real part of $E$, which is a real vector bundle.
The elements $(C,\mu)\in p_2^{-1}([\omega])$ naturally correspond to 
isomorphisms between $\PP(\RR E_{[\omega]})$ and a fixed real projective plane.
Therefore we can identify $U$ with the projective frame
bundle of $\RR E$, which clearly is a $\PGL_3(\RR)$-principal bundle.
\end{proof}

\newpage
\section{Homological quasi-simplicity}\label{sec:quasisimple}

Our aim in this section is to prove the following proposition.

\begin{prop}\label{prop:reflections}
For real irreducible nodal sextics without real nodes,
the rigid isotopy class is determined by the homological type.
\end{prop}

By Proposition~\ref{prop:period_ri}, the rigid isotopy classes of curves
with homological type $(\KL,\HTp)$ are in bijection
with the connected components of $\OHp^0$ modulo the action of the group $\GHp$.
Hence, to show that there is only one rigid isotopy class for a given homological type,
we only need to show that $\GHp$ acts transitively on the connected components of $\OHp^0$.

\subsubsection*{Walls in the period space}
We define lattices $K_+$ and $K_-$ by
\[
K_\pm = \{x\in \KL \mid x\cdot h=0, x\orth\sigma, \phi(x) = \pm x\}.
\]
Both $K_+$ and $K_-$ have one positive square; let $\mathbb{H}_+$ and $\mathbb{H}_-$
be the hyperbolic spaces associated with them, \ie
$\mathbb{H}_\pm = \{ x\in K_\pm\otimes\RR \mid x^2 > 0 \}/\RR^\ast$.
For an element $[\omega]\in\OHp$, we can write $\omega = \omega_+ + \I\,\omega_-$ with $\omega_+,\omega_-\in\KL_\RR$.
The definition of $\OHp$ implies that $\omega_+\in K_+\otimes\RR$, $\omega_-\in K_-\otimes\RR$ and $\omega_+^2 = \omega_-^2 > 0$.
The map $\OHp \to \mathbb{H}_+ \times \mathbb{H}_-$ sending
$[\omega]$ to the pair $\left([\omega_+], [\omega_-]\right)$ is a trivial two-sheeted covering.
The two connected components of $\OHp$ are interchanged by the automorphism $-\phi\in \GHp$.
Therefore, instead of studying the connected components of $\OHp^0$, we may
equivalently study their images in $\OHp/\{1,-\phi\}\cong \mathbb{H}_+ \times \mathbb{H}_-$.

Recall that the real period space $\OHp^0$ is obtained from $\OHp$ by removing
the orthogonal complements $H_\delta$ of all vectors $\delta$ in
$\Delta = \{\delta\in\KL\mid \delta^2=-2,\delta\cdot h=0, \delta\not\in\Sigma\}$.
Consider such a vector $\delta\in\Delta$.
We write it as
$\delta = \delta_\Sigma + \delta_+ + \delta_-$,
where $\delta_\Sigma \in \Sigma \otimes \QQ$,
 $\delta_+\in K_+\otimes\QQ$ and $\delta_- \in K_- \otimes \QQ$.
 An element $[\omega_+ + \I\,\omega_-]\in\OHp$ is orthogonal to $\delta$
 if and only if both $\omega_+\cdot\delta_+ = 0$ and $\omega_-\cdot\delta_- = 0$.

Let $H_{\delta_\pm} = \{x\in\mathbb{H}_\pm \mid x\cdot \delta_\pm = 0\}$.
We have 
\[
H_{\delta_\pm} = \begin{cases}
\mathbb{H}_\pm & \text{if } \delta_\pm = 0, \\
\text{a real hyperplane in $\mathbb{H}_\pm$} & \text{if } \delta_\pm^2 < 0, \\
\text{empty}                 & \text{otherwise,}
\end{cases}
\]
and $H_\delta = \{[\omega_+ + \I\,\omega_-]\in\OHp \mid ([\omega_+],[\omega_-]) \in H_{\delta_+} \times H_{\delta_-}\subset \mathbb{H}_+ \times \mathbb{H}_- \}$.
Therefore, the subspace $H_\delta$ can be either empty or of codimension one or two in $\OHp$.
Since we are only interested in the connected components of the period space, 
we may disregard the walls of codimension two.
The only relevant walls are those defined by a vector $\delta$
for which one of $\delta_+$ and $\delta_-$ is zero and the other one has a negative square.

\subsubsection*{Reflections}
Let $L$ be a non-degenerate even lattice, and let $v\in L$ be a primitive vector with $v^2 = -2k\neq 0$.
Consider the hyperplane $H_v\subset L\otimes\RR$ orthogonal to $v$, and let
$R_v$ be the reflection in this hyperplane.
It is defined by the formula
\[
R_v(x) = x + \frac{x\cdot v}{k} v.
\]
The reflection $R_v$ maps the lattice $L$ to itself if and only if
$x\cdot v\in k\ZZ$ for all $x\in L$. 
In particular, this is always the case if $k=1$, that is, $v^2 = -2$.

\begin{proof}[Proof of Proposition~\ref{prop:reflections}]
It is sufficient to show the following:
\emph{For each $\delta\in\Delta$ such that $H_\delta\subset\OHp$ is a real hyperplane,
there is an automorphism $\rho\in\GHp$ which acts on $\OHp$ as a reflection in $H_\delta$.}
Let $\delta = \delta_\Sigma + \delta_+ + \delta_-\in\Delta$.
As discussed above, $H_\delta$ is a real hyperplane in $\OHp$ if and only if
one of $\delta_+$ and $\delta_-$ is zero and the other one has a negative square;
let $\epsilon\in\{\pm 1\}$ be such that 
$\delta_\epsilon^2<0$ and $\delta_{-\epsilon} = 0$.
Let us write $\delta = \frac{1}{2}(v + w)$ where $v\in\Sigma$
and $w$ is orthogonal to $S$, with $\phi(w)=\epsilon w$ and $w^2<0$.
Since $v^2$ and $w^2$ are both non-positive and $v^2 + w^2 = 4\delta^2 = -8$,
we have $v^2\in\{0,-2,-4,-6\}$.
Therefore, the vector $v\in\Sigma$ is either zero or a sum of up to three elements of $\pm\sigma$.
We label the nodal classes $\sigma=\{s_1',s_1'',\dots,s_m',s_m''\}$, such that
$\phi(s_i') = -s_i''$ and $\phi(s_i'') = -s_i'$.
Since $\delta-\epsilon\phi(\delta) = \frac{1}{2}(v-\epsilon\phi(v))\in L$,
a root $s'_i$ appears in $v$ if and only if $s_i''$ appears as well.
This leaves two possibilities:
either $v=0$, or $v = \pm\,s_k' \pm s_k''$ for some $k$.
In the first case ($v=0$) we have $\delta = w$, and
$\rho = R_\delta$ is an automorphism of $(\KL,\HTp)$ which acts on
$\OHp$ by reflection in $H_\delta$.
Let us consider the second case, where $v = \pm\,s_k' \pm s_k''$ for some $k$.
Since changing the signs in front of $s'_k$ and $s''_k$ gives another vector in $\Delta$
defining the same hyperplane $H_\delta$, we may assume $v = s_k' + \epsilon s_k''$.
Then $\delta$ and $\phi(\delta)$ are orthogonal to each other, and
$R_\delta\circ R_{\phi(\delta)} = R_w\circ R_{v}$ is an automorphism of $\KL$
which acts on $\OH$ by reflection in $H_\delta = H_w$.
However, it does not necessarily map the set $\sigma$ to itself.
More precisely, it acts on $\Sigma$ by reflection in the hyperplane orthogonal to $v = s_k' + \epsilon s_k''$.
To ensure that the set $\sigma$ is mapped to itself, we compose $R_\delta\circ R_{\phi(\delta)}$
with a rotation in the plane $\br{s_k',s_k''}$ if $\epsilon=1$.
Indeed, 
\[\rho:=\begin{cases}
R_{s_k'}\circ R_{s_k''}\circ R_\delta\circ R_{\phi(\delta)} & \text{ if }\epsilon = +1,\\
R_\delta\circ R_{\phi(\delta)}& \text{ if }\epsilon=-1
\end{cases}\]
is an automorphism of $(\KL,\HTp)$ which acts on $\OHp$ by reflection in $H_\delta$.
\end{proof}

\section{Deducing topological information from the homological type}\label{sec:topol}

In this section we study how the homological type of a real sextic encodes
topological information such as the dividing type and the number of crossing pairs.

\subsection{Dividing type}\label{sec:topol.dividing}

Recall that an element $\alpha$ of a non-degenerate lattice $L$ is called \emph{characteristic}
if we have $x^2\equiv\alpha\cdot x\pmod{2}$ for all elements $x\in L$.
Characteristic elements always exist, and they are uniquely defined up
to $2L$, \ie the residue $\bar\alpha\in L/2L$ does not depend on the choice of $\alpha$.
A lattice is even if and only if zero is a characteristic element.
Given an involution $\phi$ on a lattice $L$,
we may consider the \emph{twisted bilinear form} given by
$(x,y)\mapsto x\cdot\phi(y)$.
We say that an element $\alpha$ is characteristic for the involution $\phi$
if it is characteristic for the corresponding twisted bilinear form.

Let $(X,h)$ be the weakly polarized K3 surface obtained from a real
irreducible nodal sextic $C\subset\CP^2$,
and let $c_1$  and $c_2$ be the real structures on $X$ which lift the complex conjugation of $\CP^2$.
Let $\alpha_1$ and $\alpha_2\in\HL{X}$ be characteristic elements for the induced
involutions $c_1^\ast$ and $c_2^\ast$ on $\HL{X}$,
and let $\bar\alpha_1$ and $\bar\alpha_2$
denote their residues in $H^2(X;\ZZ/2\ZZ)$.
The classes $\bar\alpha_i$ are Poincaré dual to the
$\ZZ/2\ZZ$-fundamental classes of the fixed-point sets $X^{c_i}$.

Note that if the real part of $C$ is non-empty, then
$C$ is dividing if and only if
the class of $\RR \tilde C$ is trivial in $H_1(\tilde C,\ZZ/2\ZZ)$.
By Proposition~\ref{prop:irred}, the sublattice $S\subset\HL{X}$ is primitive,
so that we have a natural embedding $S/2S \hookrightarrow H^2(X;\ZZ/2\ZZ)$.

\begin{prop}\label{prop:dividing}
The elements $\bar\alpha_1,\bar\alpha_2\subset H^2(X;\ZZ/2\ZZ)$ are contained in
the image of $S/2S$ if and only if $[\RR \tilde C]$ is trivial in $H_1(\tilde C;\ZZ/2\ZZ)$.
\end{prop}

To prove this proposition, we need the following lemma.
Recall that $V$ is the blow-up of $\CP^2$ in the nodes of $C$ and
that $\eta:X\to V$ is the double covering branched along $\tilde C$,
with deck transformation $\tau$.
The \emph{transfer map} $\tr: H_2(V,\tilde C;\ZZ/2\ZZ)\to H_2(X;\ZZ/2\ZZ)$
is obtained by mapping a relative 2-cell in $V$ to the sum of its two preimages in $X$.

\begin{lemma}\label{lemma:transfer}
The transfer map $\tr: H_2(V,\tilde C;\ZZ/2\ZZ)\to H_2(X;\ZZ/2\ZZ)$ is injective.
\end{lemma}

\begin{proof}
In the following, we use $\ZZ/2\ZZ$ coefficients for homology and cohomology unless
stated otherwise.
Consider the Smith exact homology sequence (see~\cite{DIK}*{p.~3})
for the pair $(X,\tau)$.
\begin{equation*}\label{eq:smith}
\cdots \to H_3(X) \xrightarrow{\eta_\ast} H_3(V,\tilde C) \xrightarrow{\Delta}
H_2(V,\tilde C) \oplus H_2(\tilde C) \xrightarrow{\tr + i_\ast} H_2(X) \to \cdots
\end{equation*}
Here, $i_\ast$ is the homomorphism induced by the inclusion $\tilde C\hookrightarrow X$.
The second component of the connecting homomorphism $\Delta$ is the boundary map
$\partial:H_3(V,\tilde C)\to H_2(\tilde C)$
which is injective because its kernel coincides with the image of $H_3(V)=0$.
To show that the transfer map is injective, suppose $x\in H_2(V,\tilde C)$ lies
in its kernel. Then $(x,0)$ lies in the kernel of $\tr +\,i_\ast$,
which coincides with the image of $\Delta$. Because $\Delta$ is injective on the second component,
this implies $x=0$.
\end{proof}

\begin{proof}[Proof of Proposition~\ref{prop:dividing}]
Consider the following diagram, where the bottom row is part of
the long exact homology sequence for the pair $(V,\tilde C)$:
\begin{center}\begin{tikzcd}
    & H_2(X)\\
  \cdots\to H_2(V) \arrow{r}{\rho} &
  H_2(V,\tilde C) \arrow{r}{\partial} \arrow[hook]{u}{\tr} &
  H_1(\tilde C) \to \cdots
\end{tikzcd}\end{center}
Note that the composition $\tr\circ\rho$ corresponds to
$\eta^\ast:H^2(V)\to H^2(X)$ if we identify homology and cohomology via Poincaré duality.
A basis for $H_2(V)$ is given by the strict transform of a line in $\CP^2$ and
the exceptional classes of the blow-up.
Therefore the image of $\eta^\ast$ is generated by the mod 2 residues
of $h$ and $\sigma$, and the image of $\eta^\ast$ is $S/2S \subset H^2(X)$.
Therefore, $\bar\alpha_i$ is contained in $S/2S$ if and only if
$[X^{c_1}]$ is contained in $\im (\tr\circ\rho)$.
Let $B_1$ and $B_2$ be the (not necessarily connected) regions
delimited by $\RR\tilde C$ in $\RP^2$.
Recall that the fixed point sets $X^{c_i}$ are double coverings of the regions $B_i$ under $\eta$,
and hence we have $[X^{c_i}] = \tr([B_i])$.
Suppose that $[\RR\tilde C]$ is trivial in $H_1(\tilde C)$.
Since $\partial [B_i] = [\RR\tilde C] = 0$ in $H_1(\tilde C)$,
the classes $[B_i]$ lie in the image of $\rho$.
Therefore, $[X^{c_i}] = \tr([B_i])$ is contained in $\im (\tr\circ\rho)$
and $\bar\alpha_i$ is contained in $S/2S$.
Now suppose that $\bar\alpha_1$ and $\bar\alpha_2$ are contained in
$S/2S$. Then the classes $[X^{c_i}]$ lie in the image of $\tr\circ\rho$. 
Because the transfer map is injective by Lemma~\ref{lemma:transfer},
this implies that the classes $[B_i]$
are in the image of $\rho$, and therefore $\partial[B_i]=[\RR\tilde C]=0$
in $H_1(\tilde C)$.
\end{proof}

\begin{remark}\label{rk:charinv}
If $C$ is dividing, then $\bar\alpha_1$ and $\bar\alpha_2$ lie in fact in
the $c^\ast_i$-invariant part of $S/2S$.
If $C$ has only non-real nodes, with classes $s_1',s_1'',\dots,s_m',s_m''$,
then the $c^\ast_i$-invariant part of $S/2S$ is generated by the mod 2 residues of
$h, s'_1+s''_1,\dots,s_m'+s_m''$.
\end{remark}

\begin{lemma}
If $C$ is an irreducible sextic with only non-real nodes,
then $\bar\alpha_1+\bar\alpha_2=\bar h$.
\end{lemma}
\begin{proof}
Note that $\RR V\subset V$ is homologous to the preimage of a general line in $\CP^2$.
Indeed, $\RR V$ intersects the preimage of a non-real line in one point and
does not intersect any of the exceptional curves since all the nodes are non-real.
Therefore, $\eta^\ast([\RR V]^{\PD}) = \bar h$, where $\PD$ denotes Poincaré duality
and $\bar h\in H^2(X)$ is the mod 2 residue of $h$.
As noted in the proof of Proposition~\ref{prop:dividing},
the map $\eta^\ast$ corresponds to $\tr\circ\rho$ under Poincare duality.
Hence, we have
\begin{align*}
\bar h &= \eta^\ast\left([\RR V]^{\PD}\right) = (\tr\circ\rho\,[\RR V])^\PD%
= (\tr([B_1] + [B_2]))^\PD \\
&= \tr([B_1])^\PD + \tr([B_2])^\PD%
=[X^{c_1}]^\PD + [X^{c_2}]^\PD = \bar\alpha_1 + \bar\alpha_2.\qedhere
\end{align*}
\end{proof}

\begin{corollary}
If $C$ is a dividing real irreducible sextic with only non-real nodes,
then exactly one of the classes $\bar\alpha_1$ and $\bar\alpha_2$ is
contained in $\Sigma/2\Sigma$.
\end{corollary}

\begin{lemma}\label{lem:propstern}
Let $C$ be a dividing real irreducible sextic with only non-real nodes.
The real structure with $\bar\alpha\in\Sigma/2\Sigma$ is
the one with property~\textup{(\ref{stern})} (see \S~\ref{sec:periods.real}).
\end{lemma}

\begin{proof}
Suppose on the contrary that the real structure $c$ with
property~(\ref{stern}) has $\bar\alpha\not\in\Sigma/2\Sigma$.
Then we can write $\bar\alpha = \bar h + y$ for some $y\in\Sigma/2\Sigma$.
Let $x\in\HL{X}$ be an element such that $c^\ast(x) = -x$ and $x\cdot h\equiv 1\pmod{2}$. 
By Remark~\ref{rk:charinv}, the element $y$ is a sum of elements of the form
$s'_i + s''_i$.
We have $x\cdot s'_i = c^\ast(x)\cdot c^\ast(s_i') = x\cdot s_i''$, and hence
$x\cdot(s'_i+s''_i)\equiv 0\pmod{2}$.
This leads to
$0\equiv -x^2 = x\cdot c^\ast(x) \equiv x\cdot \alpha \equiv x\cdot h\equiv 1\pmod{2}$,
a contradiction.
\end{proof}

\begin{remark}
Lemma~\ref{lem:propstern} also holds for curves with empty real part.
The real structure with property~(\ref{stern}) is the one
without real points, and therefore $\bar\alpha = 0$.
\end{remark}

\subsection{Perturbation of nodes}\label{sec:topol.picard}

In this subsection we examine how the homological data of a nodal sextic
changes (or rather, does not change) when a pair of non-real nodes is perturbed.
As a corollary, we show how crossing pairs can be distinguished from non-crossing pairs
on the homological level.

\subsubsection*{The complex case} Let $X_0$ be a compact complex surface
whose only singularity is an ordinary double point at $p\in X_0$.
Let $\bl_p: Y_0\to X_0$ be the blow-up of $X_0$ at $p$.
Then the induced homomorphism $\bl_p^\ast : \HL{X_0}\to\HL{Y_0}$ is
injective and its image is the orthogonal complement of $s\in\HL{Y_0}$,
the class of the exceptional curve of the blow-up.
The self-intersection of $s$ is $-2$.

Let $f:\mathcal{X}\to\mathbb{D}$ be a Lefschetz deformation of $X_0$ over the unit disk $\mathbb{D}\subset\CC$.
By this we mean that $\mathcal{X}$ is a 3-dimensional complex manifold
with $f^{-1}(0) = X_0$, and $f$ is a proper surjective holomorphic map whose
differential is non-zero everywhere except at $p\in X_0\subset\mathcal{X}$.

The space $\mathcal{X}$ retracts by deformation to the special fibre $X_0$.
Indeed, a deformation retraction $r:\mathcal{X}\to X_0$ can be obtained by
choosing a metric on $\mathcal{X}$ and using the gradient flow of the function $|f|^2$.

Let $X_\epsilon = f^{-1}(\epsilon)$ be a general fibre of $f$,
and let $i \colon X_\epsilon\hookrightarrow\mathcal{X}$ denote its inclusion.
In the following proposition we collect some facts from Picard-Lefschetz theory.
A proof of these facts can be found for instance in the book of C.~Voisin~\cite{Voisin2}*{ch. 3}.
\begin{prop}\ 
\begin{enumerate}[label=\textup{(\arabic*)}]
\item The kernel of $i_\ast:H_2(X_\epsilon;\ZZ)\to H_2(\mathcal{X};\ZZ)$ is
generated by the class of a $2$-dimensional sphere in $X_\epsilon$, called a \emph{vanishing cycle}.
Let $\delta\in\HL{X_\epsilon}$ be the class Poincaré dual to a vanishing cycle.
It is well-defined up to sign and has self-intersection number $-2$.
\item The monodromy action on $\HL{X_\epsilon}$ is given by the Picard-Lefschetz map:
\begin{align*}
R_\delta : \HL{X_\epsilon} &\to \HL{X_\epsilon} \\
 x &\mapsto x + (x,\delta)\cdot \delta
\end{align*}
\item The map $i^\ast:\HL{\mathcal{X}}\to\HL{X_\epsilon}$ is injective
and its image consists of the classes invariant under $R_\delta$, \iec the orthogonal complement of $\delta$.
\end{enumerate}
\end{prop}

\begin{prop}\label{prop:lefschetz}
There are two isometries $\psi:\HL{Y_0}\to\HL{X_\epsilon}$
which make the following diagram commute.
One maps $s$ to $\delta$, the other one to $-\delta$.
\begin{center}
\begin{tikzcd}
\HL{Y_0}\arrow{r}{\psi} &  \HL{X_\epsilon} \\
\HL{X_0}\arrow{u}{\bl_p^\ast} \arrow{r}{r^\ast} &
\HL{\mathcal{X}} \arrow{u}{i^\ast}
\end{tikzcd}
\end{center}
\end{prop}
\begin{proof}
The composition $i^\ast\circ r^\ast\circ (\bl_p^\ast)^{-1}$ defines
an isomorphism between the orthogonal complements
$\langle s\rangle^\bot\subset\HL{Y_0}$
and $\langle\delta\rangle^\bot\subset\HL{X_\epsilon}$.
It can be extended to an isomorphism $\psi:\HL{Y_0}\to\HL{X_\epsilon}$
by gluing it to an isomorphism between $\br{s}$ and $\br{\delta}$.
Such an isomorphism must map $s$ either to $\delta$ or to $-\delta$.
\end{proof}
\begin{remark}
This observation can also be obtained by explicitly
constructing the two possible resolutions of the deformation
$f:\mathcal{X}\to\mathbb{D}$, cf.~M.~Atiyah~\cite{Atiyah}.
\end{remark}

\begin{remark}\label{rk:severalnodes}
If we perturb several nodes simultaneously, the situation is very similar.
Let $X_0$ be a surface with $n$ nodes at $p_1,\dots,p_n$,
let $s_1,\dots,s_n\in\HL{Y_0}$ be the exceptional classes of the blow-up,
and let $\delta_1,\dots,\delta_n\in\HL{X_\epsilon}$ denote the classes of the corresponding vanishing cycles.
If the sublattice generated by $s_1,\dots,s_n$ is primitive in $\HL{Y_0}$,
then there are $2^n$ different isomorphisms $\psi$ making the above diagram commute.
By the gluing condition, $s_i$ must be mapped either to $\delta_i$ or $-\delta_i$,
and the signs can be chosen independently.
\end{remark}

\subsubsection*{The real case}

If the surface $X_0$ is equipped with a real structure $c:X_0\to X_0$,
then we may choose the Lefschetz deformation to be real, \iec so that $c$
extends to a real structure $c:\mathcal{X}\to\mathcal{X}$
which lifts the complex conjugation on $\mathbb{D}$.
In  this situation, we take the general fibre $X_\epsilon$
over a real point $\epsilon\in\mathbb{D}^\ast\cap\RR$,
so that $c$ defines a real structure on $X_\epsilon$.
The maps $r^\ast, i^\ast$ and $\bl_p^\ast$ above are equivariant with respect to
the real structures.

\begin{prop}[Perturbation of a real node]
Let $X_0$ be a real surface with a real node.
Choose a real deformation $f:\mathcal{X}\to\mathbb{D}$,
and let $\psi:\HL{Y_0}\to\HL{X_\epsilon}$ be an isomorphism as in Proposition~\ref{prop:lefschetz}.
If $\phi_0:\HL{Y_0}\to\HL{Y_0}$ and $\phi_\epsilon:\HL{X_\epsilon}\to\HL{X_\epsilon}$
denote the involutions induced by the respective real structures, 
then we have
\[
\phi_\epsilon = \begin{cases}\psi\circ\phi_0\circ\psi^{-1}\qquad\text { or}\\
\psi\circ\phi_0\circ\psi^{-1} \circ R_\delta\end{cases}
\]
depending on the sign of $\epsilon$, where $R_\delta$ denotes the
Picard-Lefschetz map.
\end{prop}
\begin{proof}
Since all the involved maps are equivariant with respect to the real structures,
it follows from the commutative diagram above that $\phi_\epsilon$ agrees with $\psi\circ\phi_0\circ\psi^{-1}$
on the orthogonal complement of $\delta$.
Therefore $\phi_\epsilon$ acts on $\br{\delta}$ as $\pm1$.
It remains to show that this sign really depends on the sign of $\epsilon$.
To see this, consider a path $\gamma$ connecting $\epsilon$ with $-\epsilon$
in the upper half plane (see Figure~\ref{fig:gamma}) and
let $M(\gamma) : \HL{X_\epsilon} \to \HL{X_{-\epsilon}}$ be the monodromy along $\gamma$.
Because $f$ is real (that is, $f\circ c = \conj\circ f$), we have
$M(\conj\circ\gamma^{-1}) = \phi_\epsilon\circ M(\gamma^{-1})\circ\phi_{-\epsilon}$. 
Moreover, since the concatenation of the paths
$\gamma$ and $\conj\circ\gamma^{-1}$ makes a full turn around $0$,
the composition $M(\conj\circ\gamma^{-1})\circ M(\gamma)$ is the Picard-Lefschetz map $R_\delta$.
It follows that $\phi_\epsilon = M(\gamma)^{-1}\circ\phi_{-\epsilon}\circ M(\gamma)\circ R_\delta$.
Therefore, if $\phi_\epsilon$ maps the vanishing cycle $\delta_\epsilon$
to $a\delta_\epsilon$, $a\in\{\pm 1\}$,
then $\phi_{-\epsilon}$ maps the vanishing cycle $\delta_{-\epsilon}$
to $-a\delta_{-\epsilon}$.
\begin{figure}[h]
\begin{tikzpicture}
\draw (-1.8,0) -- (1.8,0);
\draw (0,0) circle (1.3cm);
\draw[-latex] (0.7,0) arc (0:173:0.7cm);
\draw[-latex] (-0.7,-0) arc (180:353:0.7cm);
\draw (0.7,0) node [vec,label=-45:\footnotesize$\epsilon\phantom{-}$] {};
\draw (-0.7,0) node [vec,label=-135:\footnotesize$-\epsilon$] {};
\draw (0,0) node [vec,label=-90:\footnotesize$0$] {};
\draw (0,0.7) node[anchor=south,inner sep=2pt] {\footnotesize$\gamma$};
\draw (0,-0.7) node[anchor=north,inner sep=1pt] {\footnotesize$\conj\circ\gamma^{-1}$};
\draw (0.6,0.5) node[anchor=south west] {\footnotesize$\mathbb{D}$};
\draw (1.5,0) node[anchor=north west] {\footnotesize$\mathbb{R}$};
\end{tikzpicture}
\caption{}\label{fig:gamma}
\end{figure}
\end{proof}

\begin{prop}[Perturbation of a pair of non-real nodes]\label{prop:deform}
Let $X_0$ be a real surface with a pair of non-real nodes.
Choose a real deformation $f:\mathcal{X}\to\mathbb{D}$ of this pair.
Then there is an isomorphism $\psi:\HL{Y_0}\to\HL{X_\epsilon}$ 
such that $\phi_\epsilon = \psi\circ\phi_0\circ\psi^{-1}$.
\end{prop}
\begin{proof}
Let $s',s''\in\HL{Y_0}$ denote the classes of the exceptional curves,
and let $\delta',\delta''\in\HL{X_\epsilon}$ be the corresponding vanishing cycles.
By Proposition~\ref{prop:lefschetz} and Remark~\ref{rk:severalnodes},
there is an isomorphism $\psi:\HL{Y_0}\to\HL{X_\epsilon}$
with $\psi(s') = \delta'$, $\psi(s'') = \delta''$, and
such that $\phi_\epsilon$ agrees with $\psi\circ\phi_0\circ\psi^{-1}$
on the orthogonal complement of $\br{\delta',\delta''}$.
Moreover, $\psi\circ\phi_0\circ\psi^{-1}$ maps $\delta'$ to $-\delta''$,
and $\phi_\epsilon$ must map $\delta'$ either to $-\delta''$ or to $\delta''$.
In the former case we are done, and in the latter case it suffices to
replace $\psi$ by $R_{\delta'}\circ\psi$.
\end{proof}

\begin{corollary}
Let $C$ be a real irreducible sextic whose only singularities are non-real nodes,
and let $C'$ be a real non-singular sextic obtained from $C$ be perturbing
all the nodes. Let $(X,h,c)$ and $(X',h',c')$ be the weakly polarized
real K3 surfaces obtained from $C$ and $C'$ respectively.
Then there is an isomorphism $\psi:\HL{X}\to\HL{X'}$ such that $\psi(h) = h'$
and $\psi\circ c^\ast = c'^\ast\circ\psi$.
\end{corollary}

\begin{corollary}\label{cor:crossing}
Let $C$ be a dividing real irreducible sextic whose only singularities
are non-real nodes,
let $(X,h,c)$ be the associated weakly polarized real K3 surface,
let $\alpha\in\HL{X}$ be a characteristic element of $c^\ast$ and let
$\sigma=\{s_1', s_1'', \dots, s_m', s_m''\}$ denote its set of nodal classes.
Then we have
$\alpha \equiv \sum_{i\in I} s_i' + s_i'' \mod{2}$,
where $I\subset\{1,\dots,m\}$ is the set of indices of the crossing pairs.
\end{corollary}
\begin{proof}
By Remark~\ref{rk:charinv} and Lemma~\ref{lem:propstern}, we can write
$\alpha \equiv \sum_{i\in I} s_i' + s_i'' \mod{2}$
for some set $I\subset\{1,\dots,m\}$.
To see whether a particular $i\in\{1,\dots,m\}$ belongs to $I$, consider
a curve $C'$ obtained from $C$ by perturbing the pair of nodes in question
while keeping all the other nodes.
Let $(X',h',c')$ be the corresponding weakly polarized real K3 surface,
with set of nodal classes $\sigma'$.
Let $\psi:H^2(X;\ZZ)\to H^2(X';\ZZ)$ be an isomorphism as in Proposition~\ref{prop:deform}.
A characteristic element for $c'^\ast$ is given by $\alpha':=\psi(\alpha)$.
The $i$-th pair of non-real nodes is crossing if and only if $C'$
is not dividing.
By Proposition~\ref{prop:dividing}, this is the case if and only if
$\alpha'$ is not contained (modulo 2) in the subgroup generated by $\sigma'$,
which is equivalent to $\alpha$ not being contained (modulo 2) in the subgroup
generated by $\psi^{-1}(\sigma') = \sigma\setminus\{s_i',s_i''\}$.
This happens if and only if $i\in I$.
\end{proof}

\section{Classification of the homological types}\label{sec:classif}

\subsection{Discriminant groups and discriminant forms}

The notions of \emph{discriminant group} and \emph{discriminant form}
of a lattice play an important role in the classification of the homological types.
In the following we recall the most important definitions;
all the necessary details can be found in Nikulin's articles \cite{Nikulin79} and \cite{Nikulin83}.

A \emph{lattice} is a finitely generated free abelian group $L$
endowed with an integral symmetric bilinear form $b\colon L\times L\to\ZZ$.
We usually write $x\cdot y$ instead of $b(x,y)$ and $x^2$ instead of $b(x,x)$.
A lattice $L$ is \emph{even} if $x^2$ is even for all $x\in L$.
The \emph{correlation homomorphism} of a lattice $L$ is the homomorphism
$L\to L^\ast = \Hom(L,\ZZ)$ defined by $x\mapsto(y\mapsto x\cdot y)$.
A lattice is called \emph{non-degenerate} if its correlation homomorphism is
injective, and \emph{unimodular} if its correlation homomorphism is
bijective. The \emph{discriminant group} of a non-degenerate lattice $L$
is  $A_L=L^\ast/L$, the cokernel of its correlation homomorphism.
The discriminant group is a finite abelian group.
We may view $L^\ast$ as a subgroup of $L\otimes\QQ$;
therefore $L^\ast$ inherits a $\QQ$-valued bilinear form.
This  induces a bilinear form $b_L:A_L\times A_L\to\QQ/\ZZ$,
called the \emph{discriminant bilinear form} of $L$.
If $L$ is even, then the squares of elements in $A_L$ are well-defined modulo $2\ZZ$,
\ie we have a finite quadratic form $q_L:A_L\to\QQ/2\ZZ$, called the
\emph{discriminant (quadratic) form} of $L$.
Note that the discriminant bilinear form can be recovered from the
discriminant quadratic form using the identity
$b_L(x,y) = \frac{1}{2}(q_L(x+y) - q_L(x) - q_L(y))$.

\subsection{Statement of the classification}\label{sec:classif.statement}

We fix the number $m\ge 1$ of pairs of non-real nodes.
(The case $m=0$ corresponds to non-singular sextics, for which
the homological types are classified by Nikulin~\cite{Nikulin79}.
We exclude it here to avoid dealing with certain boundary conditions.)
As in the previous sections, we consider a triple $(\KL,h,\sigma)$,
where $\KL$ is a fixed K3 lattice,
$h\in\KL$ is a fixed polarization vector with $h^2 = 2$,
and $\sigma\subset L$ is a set of $2m$ pairwise orthogonal roots
orthogonal to $h$, such that the sublattice $S$ generated by $h$ and $\sigma$
is primitive in $\KL$.
By Proposition~\ref{prop:uniquetriple}, all such triples are isomorphic to each other.

Recall that a \emph{geometric involution} on $(\KL,h,\sigma)$
is an isometric involution $\phi:\KL\to\KL$
sending $h$ to $-h$ and $\sigma$ to $-\sigma$ (as a set),
such that the lattice $\KL_+ = \{x\in\KL\mid\phi(x)=x\}$
has one positive square, and $\phi$ satisfies condition $(\ast)$ (see \S~\ref{sec:periods.real}).
Two such involutions, say $\phi_1$ and $\phi_2$, are called \emph{equivalent}
if there is an automorphism $\rho:\KL\to\KL$ with
$\rho(h)=h$ and $\rho(\sigma)=\sigma$ such that $\phi_2\circ\rho = \rho\circ\phi_1$.
Our aim is to classify homological types $(\KL,h,\sigma,\phi)$ up to this equivalence.

\begin{adhocrmk}{Convention}
In this section, whenever we speak about a geometric involution $\phi$,
we assume that it corresponds to sextics without real nodes,
\ie  there is no root $s\in\sigma$ with $\phi(s)=-s$.
In other words, we use ``geometric involution'' as a shorthand for
``geometric involution corresponding to real irreducible sextics whose only singularities
are $m$ pairs of non-real nodes'', and likewise for ``homological type''.
\end{adhocrmk}

\begin{definition}\label{def:inv}
We define invariants $a,t,\delta$ and $r$ for a homological type $(\KL,\HTp)$ as follows.
\begin{description}[labelwidth=*,font=\normalfont,align=right,leftmargin=2em]
\item[$a\;$] Since $\phi$ is an involution on a unimodular lattice,
the discriminant group of the invariant sublattice $\KL_+$, which we denote by $A_q$, is of period 2.
Let $a$ be the length of this group. In other words, we have $A_q \cong (\ZZ/2\ZZ)^a$.
\item[$t\;$] Let $t$ be the number of negative squares of the invariant lattice $\KL_+$.
\item[$\delta\;$]Let $\alpha\in\KL$ be a characteristic element of the involution
$\phi$, and let $\bar\alpha$ be its residue in $\KL/2\KL$ (see \S~\ref{sec:topol.dividing}).
Recall that $S/2S$ is naturally embedded in $\KL/2\KL$.
Let $\delta = 0$ if $\bar\alpha \in S/2S$ and $\delta = 1$ otherwise.
\item[$\ncr\;$] The invariant $\ncr$ is only defined if $\delta = 0$.
If we label the nodal classes $\sigma = \{s_1',s_1'',\dots,s_m',s_m''\}$
such that $\phi(s_i') = -s_i''$ for $i=1,\dots,m$,
then the characteristic element $\alpha$ can be expressed as a sum
\[
\alpha \equiv \sum_{i\in I}s_i'+s_i'' \mod{2}
\]
for some  subset $I\subset\{1,\dots,m\}$, cf. Remark~\ref{rk:charinv} and Lemma~\ref{lem:propstern}.
Let $r$ be the cardinality of $I$.
\end{description}
\end{definition}

Note that equivalent homological types have the same invariants.
      
\begin{theorem}\label{thm:unique}
Two homological types are equivalent if and only if they have the same
invariants $(a, t, \delta, \ncr)$.
\end{theorem}

\begin{theorem}\label{thm:conditions}
A homological type $(\KL,\HTp)$ with invariants $(a, t, \delta, \ncr)$
exists if and only if $a, t, \delta$ and $\ncr$ satisfy the following conditions.
\begin{enumerate}[series=thmconditions,label=\textup{(\textit{\roman*})},noitemsep]
\item \label{item:cond1} $a\le 1+t\le 20-a$.
\item \label{item:cond2} $a\equiv1+t\mod{2}$.
\item \label{item:cond3} $2m\le a$, with equality only if $\delta=0$.
\item \label{item:cond4} If $\delta = 0$, then $2\ncr\equiv 1-t\mod{4}$.
\item \label{item:cond5} If $\delta = 0$ and $a=1+t$, then $2\ncr\equiv a-2\mod{8}$.
\end{enumerate}
\end{theorem}

\subsection{Proof of Theorem~\ref{thm:unique}}\label{sec:classif.thm3}

\begin{definition}[Nikulin \cite{Nikulin83}]
A \emph{condition on an involution} is a triple $(S,\theta,G)$  formed by
an even lattice $S$, an involution $\theta:S\to S$ and
a normal subgroup $G$ of $\{g\in O(S) \mid g\circ\theta = \theta\circ g\}$.
A \emph{(unimodular) involution with condition $(S,\theta,G)$} is a triple $(L, i, \phi)$
formed by a (unimodular) even lattice $L$, a primitive embedding $i:S\hookrightarrow L$ and an involution
$\phi$ on $L$ such that $\phi\circ i = i\circ\theta$.
Two triples $(L_1,i_1,\phi_1)$ and $(L_2,i_2,\phi_2)$ are called \emph{isomorphic}
if there is an isometry $\rho:L_1\to L_2$ such that $\rho\circ\phi_1 = \phi_2\circ\rho$
and there is $g\in G$ such that $\rho\circ i_1 = i_2 \circ g$.
\end{definition}
For a fixed condition $(S,\theta,G)$,
Nikulin \cite{Nikulin83} defines invariants which uniquely determine the
genus of involutions with condition $(S,\theta,G)$.
(Roughly speaking, two triples $(L_1,\phi_1,i_1)$ and $(L_2,\phi_2,i_2)$
have the same \emph{genus} if they are isomorphic  over $\QQ$ and over
the $p$-adic numbers $\ZZ_p$ for all primes $p$.)
Moreover he describes all the values these invariants can take, and
he gives conditions under which all involutions of the same genus are isomorphic.

The classification of homological types fits into this framework.
To see this, we define a condition $(\hat{S},\theta,G)$ as follows.
To distinguish abstract lattices and vectors from their images in $\KL$,
we put a hat over the corresponding symbols.
Let $\hat{S}$ be the lattice $\left<2\right>\oplus\left<-2\right>^{2m}$ with fixed orthogonal basis
$\{\hat{h}\}\cup\hat{\sigma}$ where $\hat{h}^2=2$ and
$\hat{\sigma}=\{\os'_1,\os''_1, \dots, \os'_m, \os''_m\}$.
Let the involution $\theta:\hat{S}\to \hat{S}$ be defined by $\theta(\hat{h})=-\hat{h},\theta(\os'_i) = -\os''_i$ and $\theta(\os''_i) = -\os'_i$.
Finally, let $G = \{g\in O(\hat{S})\mid g(\hat{h})=\hat{h}\text{ and }g(\hat{\sigma} )= \hat{\sigma}\}$.
The group $G$ is a semi-direct product $(\ZZ/2\ZZ)^m\rtimes S_m$, where a permutation $f\in S_m$ takes
$\os'_i$ to $\os'_{f(i)}$ and $\os''_i$ to $\os''_{f(i)}$,
and the $i$-th basis vector $e_i\in(\ZZ/2\ZZ)^m$ acts by flipping the pair $\{\os'_i,\os''_i\}$
and leaving the other basis vectors fixed.

\begin{lemma}\label{lemma:cond}
There is a one-to-one correspondence between
geometric involutions $\phi$ on the fixed triple $(\KL, h, \sigma)$,
considered up to equivalence, and
involutions with condition $(\hat{S},\theta,G)$, say $(L,\phi, i)$,
where $L$ is isomorphic to $\KL$ and such that $\phi$ is a geometric involution on $L$, 
considered up to isomorphism.
\end{lemma}

\begin{proof}
Given a triple $(L,i,\phi)$, by Proposition~\ref{prop:uniquetriple}
there is always an isometry $\rho:L\to\KL$
such that $\rho\circ i$ maps $\hat h\in\hat S$ to the fixed vector $h\in\KL$
and $\rho\circ i(\hat\sigma)$ is the fixed set $\sigma\subset\KL$.
Moreover, isomorphic triples give rise to equivalent involutions.
To pass from a geometric involution $\phi$ on $(\KL,h,\sigma)$
to an involution with condition, one can set $L=\KL$
and choose a bijection between $\hat{\sigma}$ and $\sigma$
mapping pairs $\{\os'_i,\os''_i\}$ to pairs of roots exchanged by $-\phi$.
Such a bijection always exists and is uniquely determined up to the action of $G$.
\end{proof}

In \cite{Nikulin83}*{Theorem~1.6.3} Nikulin gives a complete system of
invariants for the genus of involutions with condition.
In order to prove Theorem~\ref{thm:unique}, it suffices to show that
the isomorphism classes are determined by the genus, and
that the invariants $m$, $a$, $t$, $\delta$ and $\ncr$ (see Definition~\ref{def:inv})
determine Nikulin's invariants for the genus.

\subsubsection*{Uniqueness in the genus}
Given a homological type $(\KL,\HTp)$ we define the following sublattices of $\KL$.
\begin{align*}
    \KL_+ &= \{x\in \KL\mid \phi(x)=x\} & \KL_- &= \{x\in \KL\mid \phi(x)=- x\} \\
    S_+ &= \KL_+\cap S      & S_- &= \KL_-\cap S\\
    K_+ &= \KL_+\cap S^\bot & K_- &= \KL_- \cap S^\bot
\end{align*}

By \cite{Nikulin83}*{Remark~1.6.2}, for an isomorphism class of an involution with condition
to be unique in its genus it is sufficient that the lattices
$K_\pm$ are unique in their genera and that the natural homomorphisms
$O(K_\pm)\to O(k_\pm)$ are surjective.

\begin{prop}\label{prop:genusunique}
If $K$ is an indefinite even lattice whose discriminant form $k$ is of period $4$,
then $K$ is unique in its genus and the homomorphism $O(K)\to O(k)$ is surjective.
\end{prop}
\begin{proof}
For the case where the length of the discriminant group $k$ is smaller
than the rank of the lattice $K$, see \cite{Nikulin79}*{Theorem~1.14.2}.
If the length of $k$ is equal to the rank of $K$ and is at least 3, 
the statement follows from a result of Miranda and Morrison
\cite{MirandaMorrison}*{chapter 8, Corollary~7.8}.
For lattices of rank 2 with $k$ of period 4, the statement can be verified by hand.
The only such lattices are
$\U,$ $\br{2}\oplus\br{-2}$, $\U(2)$, $\br{2}\oplus\br{-4}$,
$\br{4}\oplus\br{-4}$ and $\U(4)$.
\end{proof}

\begin{corollary}
Let $(\KL,\HTp)$ be a homological type.
Then the lattices $K_+$ and $K_-$ are unique in their genera, and the natural homomorphisms
$O(K_\pm)\to O(k_\pm)$ are surjective.
In particular, the homological type is unique in its genus.
\end{corollary}

\begin{proof}
Both $K_+$ and $K_-$ have one positive square and their
discriminant groups are of period 4.
Therefore they are either indefinite, and we can apply the proposition,
or they are of rank 1, and the statement is trivially verified.
For the uniqueness in the genus, see \cite{Nikulin83}*{Remark~1.6.2}.
\end{proof}

\subsubsection*{Invariants for the genus.}
In the following we define Nikulin's invariants of the genus and show that for involutions
with condition $(\hat{S},\theta,G)$ as in Lemma~\ref{lemma:cond}, they are determined by the invariants 
$m$, $a$, $t$, $\delta$ and $\ncr$ (see Definition~\ref{def:inv}).

\emph{Invariants of $(\KL,\phi)$.}
Let $A_q$ be the discriminant group of the lattice $\KL_+$, and let $q$ be the
discriminant form defined on it.
The invariants characterizing $(\KL,\phi)$ are the rank and signature of $\KL$,
the rank and signature of $\KL_+$
and the invariants of $q$.
Since $\KL$ is the fixed K3 lattice, its rank and signature are fixed.
Since $\phi$ is a geometric involution, $\KL_+$ has one positive square.
Therefore, its rank and signature are determined by $t$,
its number of negative squares.
Finally the group $A_q$ is of period 2, and hence $q$ is determined by
the length of $A_q$, which is $a$, and by the parity of $q$
(cf. \cite{Nikulin79}*{Theorem~ 3.6.2}).
The form $q$ is even if and only if the characteristic element $\bar\alpha\in \KL/2\KL$
 is zero. This is the case if and only if
$\delta=0$ and $\ncr=0$.

\emph{The groups $H_+$ and $H_-$.}
Let $A_{S_+}$ and $A_{S_-}$ be the discriminant groups of the lattices
$S_+$ and $S_-$, respectively.
Following \cite{Nikulin83}, we define subgroups $\Gamma_\pm$ and $H_\pm$
of $A_{S_\pm}$ as follows.
\begin{align*}
\Gamma_\pm &= \{x\in S_\pm^\ast\mid \exists y\in S_\mp^\ast \text{ such that }x+y\in S\}/S_\pm
        \subset S_\pm^\ast/S_\pm=A_{S_\pm} \\
H_\pm &= \{x\in S_\pm^\ast\mid \exists y\in \KL_\mp^\ast \text{ such that }x+y\in \KL\}/S_\pm
        \subset S_\pm^\ast/S_\pm=A_{S_\pm}
\end{align*}
Note that the subgroups $\Gamma_+$ and $\Gamma_-$ depend only on the condition $(\hat{S},\theta,G)$,
whereas the subgroups $H_\pm$ depend a priori on the involution $\phi$.
However, it turns out that in our case $H_+$ and $H_-$ do not depend on $\phi$.
We have $S_+ = \langle p_1, \dots, p_m\rangle$ and
$S_- = \langle n_1, \dots, n_m, h\rangle$
where $p_i = s'_i - s''_i$ and $n_i = s'_i + s''_i$.
The groups $A_{S_\pm}$ and $\Gamma_\pm$ are given as follows.
\begin{align*}
A_{S_+} &= \left< [\tfrac{p_1}{4}], \dots, [\tfrac{p_m}{4}] \right> %
    & &\cong (\ZZ/4\ZZ)^m  \\
A_{S_-} &= \left<[\tfrac{n_1}{4}], \dots, [\tfrac{n_m}{4}], [\tfrac{h}{2}]\right> %
    & &\cong (\ZZ/4\ZZ)^m \oplus \ZZ/2\ZZ  \\
\Gamma_+ &= \left< [\tfrac{p_1}{2}], \dots, [\tfrac{p_m}{2}]\right> %
    & &\cong (\ZZ/2\ZZ)^m \subset A_{S_+} \\
\Gamma_- &= \left< [\tfrac{n_1}{2}], \dots, [\tfrac{n_m}{2}]\right> %
    & &\cong (\ZZ/2\ZZ)^m \subset A_{S_-}
\end{align*}
The group $H_+$ contains $\Gamma_+$ and is contained in
the 2-torsion subgroup of $A_{S_+}$.
Since $\Gamma_+$ is the 2-torsion subgroup of $A_{S_+}$,
we have $H_+$ = $\Gamma_-$.
Similarly, $H_-$ contains $\Gamma_-$ and is contained in
the 2-torsion subgroup of $A_{S_-}$.
The group $\Gamma_-$ is of index two in the 2-torsion subgroup of $A_{S_-}$,
so $H_-$ must be either $\Gamma_-$ or $\Gamma_-\oplus \langle[\tfrac{h}{2}]\rangle$.
Because we chose the real structure with property (\ref{stern}) (see \S~\ref{sec:periods.real}),
the vector $h$ does not glue to $L_+$, hence
we have $[\tfrac{h}{2}]\not\in H_-$ and therefore $H_- = \Gamma_-$.

\emph{The finite quadratic form $q_r$.}\label{def:qr}
The lattice $S$ determines an anti-isometry $\gamma: \Gamma_+ \to \Gamma_-$.
We define the lattice $H_+ \oplus_\gamma H_-$ as the gluing of $H_+$
and $H_-$ along $\gamma$. More precisely, let
\[
H_+ \oplus_\gamma H_- = \raisebox{1pt}{$(\Gamma_\gamma)^\bot_{H_+ \oplus H_-}$}%
\,\big/\,\raisebox{-1pt}{$\Gamma_\gamma$,}
\]
where $\Gamma_\gamma \subset \Gamma_+ \oplus \Gamma_-$ is the graph of $\gamma$.
There is a natural embedding $\gamma_r : H_+ \oplus_\gamma H_- \hookrightarrow q$,
inducing a finite quadratic form $q_r$ on $H_+ \oplus_\gamma H_-$.
In general (if $\Gamma_+ \oplus_\gamma \Gamma_-$ is strictly contained in $H_+ \oplus_\gamma H_-$)
this form $q_r$ depends on the embedding $i:S\hookrightarrow L$ and on the involution $\phi$.
Since in our case we have $\Gamma_\pm = H_\pm$ however, the form $q_r$
is determined by $S$ and $\theta$.

\emph{Invariants of the embedding $\gamma_r$.}
Let $v_q\in A_q$ be the characteristic element of $q$, \ie the element for which
$(v_q,x)= x^2\mod{1}$ for all $x\in A_q$.
We have $v_q = \left[\frac{\alpha}{2}\right]\in A_q$, where
$\alpha\in \KL$ is a characteristic element for the involution $\phi$.
The embedding $\gamma_r$ is characterized by whether $v_q$ is contained in $\gamma_r(q_r)$,
and if this is the case, by the orbit $G\cdot\gamma_r^{-1}(v_q)\subset q_r$.
But $v_q$ is contained in $\gamma_r(q_r)$ if and only if $\bar\alpha\in\KL/2\KL$ is contained in 
$S/2S$, which is by definition if and only of $\delta = 0$.
In this case we have $\alpha \equiv \sum_{i\in I} s'_i+s''_i \mod{2L}$, where
the sum ranges over the indices of the crossing pairs.
Therefore, $\gamma_r^{-1}(v_q) = \sum_i [\tfrac{p_i}{2}]\in q_r$
is determined (up to the action of $G$, which permutes the indices)
by $\ncr$.

This concludes the proof of Theorem~\ref{thm:unique}.\hfill\qed

\subsection{Proof of Theorem~\ref{thm:conditions}}\label{sec:classif.thm4}
\begin{proof}[Proof of Theorem~\ref{thm:conditions}]
By Lemma~\ref{lemma:cond}, the homological types we consider correspond
to involutions with condition $(\hat{S},\theta,G)$.
Therefore we can deduce Theorem~\ref{thm:conditions} from
Nikulin's existence theorem for involutions with conditions \cite{Nikulin83}*{Theorem~1.8.3}.
To do this, we need to check that the conditions \ref{item:cond1}~--~\ref{item:cond5} of Theorem~\ref{thm:conditions}
are equivalent to the conditions given in \cite{Nikulin83}*{Conditions~1.8.1 and 1.8.2}.

For the most part, this verification is straightforward.
The only more technical step, which we consider in detail,
is the equivalence of condition \ref{item:cond5} of Theorem~\ref{thm:conditions}
with \emph{boundary condition~1} in \cite{Nikulin83}*{Conditions~1.8.2},
in the case $\delta=0$.
This boundary condition guarantees the existence of a lattice $K_+$
with prescribed discriminant quadratic form $k_+$ and index of inertia $(1,t-m)$,
using Nikulin's theorem on the existence of an indefinite even lattice
with prescribed rank, signature and discriminant form
(\cite{Nikulin79}*{Theorem~1.10.1}).
In our setting, \emph{boundary condition~1} in \cite{Nikulin83}*{Conditions 1.8.2} states that
\begin{equation}\tag{BC1}\label{eq:BC1}
\text{if }a = 1+t, \text{ then }1-t\equiv 4\,\epsilon_{v_+} + c_v \pmod{8}.
\end{equation}
where $c_v$ and $\epsilon_{v_+}$ are invariants taking values in $\ZZ/8\ZZ$ and
 $\ZZ/2\ZZ$ respectively, whose definition we give below.
We need to show that (\ref{eq:BC1}) is equivalent to condition~\ref{item:cond5}
which states that
\begin{equation*}
\text{if }a = 1+t, \text{ then }2\ncr \equiv a-2 \pmod{8}.
\end{equation*}
The equivalence between (\ref{eq:BC1}) and \ref{item:cond5}
follows from Lemma~\ref{lemma:epsilon} below.
\end{proof}

In the following, suppose that $\delta = 0$.
Recall from \S~\ref{def:qr} that $v_q$ denotes the characteristic
element of the finite quadratic form $q$.

\begin{definition}
The invariant $c_v\in\ZZ/8\ZZ$ is defined such that
$\frac{1}{2} c_v = q(v_q)\mod{2}$ (see \cite{Nikulin83}*{p.~113}) and
$c_v \in\{ -1,0,1,2\} \mod{8}$ (see \cite{Nikulin83}*{p.~118}).
\end{definition}

Before defining $\epsilon_{v_+}$, we introduce two auxiliary finite quadratic forms,
$\gamma_+$ and $\eta$.
Let $\gamma_+$ be the finite quadratic form
\[
\gamma_+ = \begin{cases}
\left[\frac{1}{2}\right] \oplus \left[-\frac{1}{2}\right] & \text{ if $\ncr$ is even,} \\
\left[\frac{1}{2}\right] \oplus \left[\frac{1}{2}\right] & \text{ if $\ncr$ is odd}
\end{cases}
\]
(see \cite{Nikulin83}*{p.~117, eq.~8.14}),
and let $x_1$ and $x_2$ be the standard generators of $\gamma_+$.
The characteristic element of $\gamma_+$ is $v_{\gamma_+} = x_1 + x_2$,
and we have $v_{\gamma_+}^2 \equiv \ncr \pmod{2}$.

Let $A_{S_+}$ be the discriminant group of the lattice $S_+$, and let
$s_+$ be the finite quadratic form defined on it.
Recall that $A_{S_+}$ is generated by the elements $\alpha_i = [\frac{s'_i-s''_i}{4}]$
for $i\in\{1,\dots,m\}$.
Since $q_r \cong \Gamma_- \subset A_{S_-}$, we may view $v = \gamma_r^{-1}(v_q)$
as an element of $A_{S_-}$ (see \S~\ref{def:qr} for the definition of $q_r$ and $\gamma_r$).
We have $v = \sum_{i\in I} 2\alpha_i$.
The finite quadratic form $\eta$ is obtained by gluing of $s_+$ and
$\gamma_+$, where $v\in s_+$ is identified with $v_{\gamma_+}\in\gamma_+$.
More precisely, let $v_\eta = v + v_{\gamma_+} \in s_+ \oplus \gamma_+$.
Note that $v_\eta^2=0$. Hence we can define the quadratic form
\[
\eta = \br{v_\eta}^\bot_{(s_+ \oplus\,\gamma_+)}/\br{v_\eta},
\]
which is denoted $(q_{S_+^2})_v$ in \cite{Nikulin83}*{p.~117, eq.~8.15}.
Let $K(\eta_2)$ be a 2-adic lattice of minimal length whose discriminant
quadratic form is $\eta$. 
The discriminant of $K(\eta_2)$ is an element
of $\ZZ_2/(\ZZ_2^\ast)^2$.
There are two such lattices $K(\eta_2)$, and their discriminants differ by a sign.

\begin{definition}
The invariant $\epsilon_{v_+} \in\ZZ/2\ZZ$ is defined by the equation
\[
5^{\epsilon_{v_+}} = \pm \discr K(\eta_2) \in {\QQ_2/(\QQ_2^\ast)^2}.
\]
\end{definition}
Note that the group $\QQ^\ast_2/(\QQ_2^\ast)^2$ is isomorphic to
$(\ZZ/2\ZZ)^3$ and has representatives $\{\pm1,\pm2,\pm5,\pm10\}$,
see for example J.-P.~Serre \cite{SerreArith}*{p.~18}.

\begin{lemma}\label{lemma:epsilon}
If $\delta = 0$, then $4\,\epsilon_{v_+} + c_v + 2\ncr \equiv 0 \mod{8}$.
\end{lemma}

\begin{proof}
Recall that $v_q = \sum_i [\tfrac{p_i}{2}]\in q_r$,
where the sum ranges over the indices of the crossing pairs.
We have $q(v_q) \equiv \ncr\mod{2}$. Therefore, 
\[
c_v \equiv \begin{cases}
0 \pmod{8} & \text{ if $\ncr$ is even,} \\
2 \pmod{8} & \text{ if $\ncr$ is odd.}
\end{cases}
\]
Hence, to prove the lemma, it remains to show that
\[\epsilon_{v_+} \equiv \begin{cases}
0 \pmod{2} & \text{if }\ncr\equiv 0,3\pmod{4},\\
1 \pmod{2}& \text{if }\ncr\equiv 1,2\pmod{4}.\\
\end{cases}
\]

Up to the action of $G$, we may suppose that $I=\{1,\dots, \ncr\}$.
Then a basis of the finite quadratic form $\eta$ defined above is given by
\[\{\alpha_i + x_1\}_{i=1}^{\ncr} \cup \{\alpha_{i}\}_{i=\ncr+1}^m.\]
We define a new basis $\{\beta_i\}_{i=1}^m$ for $\eta$ by
\[\beta_i = \begin{cases}
[2(\alpha_1 + \cdots + \alpha_{i-1}) + \alpha_i + x_1]  & \text{ if } i\in\{1,\dots,\ncr\}\\
 [\alpha_i] &\text{ if }i\in\{\ncr+1,\dots,m\}.
\end{cases}\]
Note that this basis is orthogonal, and we have
\[\beta_i^2 =  \begin{cases}
-\frac{5}{4}  & \text{ if } i\le \ncr\text{ and $i$ is odd,}\\
-\frac{1}{4}  & \text{ otherwise.}
\end{cases}\]
Now $K(\eta_2)$ can be defined as the 2-adic lattice with orthogonal basis $b_1,\dots,b_m$,
where
\[b_i^2 = \begin{cases}
-5\cdot 4  & \text{ if } i\le \ncr\text{ and $i$ is odd,}\\
-1\cdot 4  & \text{ otherwise.}
\end{cases}\]
We calculate the discriminant of $K(\eta_2)$ with respect to this basis.
Since the basis is orthogonal, the discriminant is just the product $b_1^2\cdot b_2^2\cdots b_m^2$.
The invariant $\epsilon_{v_+}$ is the number of occurrences of the factor 5 in this product,
counted modulo 2. This is equal to the number of odd integers between 1 and $\ncr$,
and it is finally easy to check that
\[
\epsilon_{v_+} \equiv \left\lceil\frac{\ncr}{2}\right\rceil
\equiv \begin{cases}
0 & \text{if }\ncr\equiv 0,3\mod{4}\\
1 & \text{if }\ncr\equiv 1,2\mod{4}
\end{cases}\pmod{2}.\qedhere
\]
\end{proof}

\section{Proof of Theorem~\ref{thm:curve-uniqueness} and Theorem~\ref{thm:curve-existence}}\label{sec:thm12}

\begin{prop}[Geometrical interpretation of $a,t,\delta$ and $r$]\label{prop:corresp}
Let $C\subset\CP^2$ be a real irreducible sextic  whose only singularities
are $m$ pairs of non-real nodes, and let $(\KL,\HTp)$ be its homological type.
Then the invariants $(a,t,\delta,r)$ have the following geometrical interpretation:
if $\RR C$ is empty, we have $a=10$, $t=9$, $\delta=0$ and $r=0$;
if $\RR C$ is not empty, we have
\begin{align*}
a &= 11-l, \text{ where $l$ is the number of ovals of $C$}, \\
t &= 9 + \chi(B), \text{ where $B$ denotes the non-orientable half of $\RP^2\setminus\RR C$}, \\
\delta &= \begin{cases}1 & \text{if $C$ is dividing and}\\
                        0&\text{otherwise,}\end{cases} \\
r &\phantom{=\;}\text{ is the number of crossing pairs of $C$}.
\end{align*}
\end{prop}
\begin{proof}
The interpretations for $a$ and $t$ are well-known for non-singular sextics,
cf. \cite{Nikulin79}*{Theorem~3.10.6}.
If $C\subset\CP^2$ is a real irreducible sextic whose only singularities
are $m$ pairs of non-real nodes, consider a real non-singular sextic
$C'\subset\CP^2$ obtained from $C$ be perturbing all the nodes.
The perturbation does not change the isotopy type of the real part,
so the right hand side of the equations
does not change when passing from $C$ to $C'$.
If the homological type of $C$ is $(\KL,h,\sigma,\phi)$,
then the homological type of the perturbed sextic $C'$ is $(\KL,h,\emptyset,\phi)$
for a suitable marking (see \S~\ref{sec:topol.picard}),
and thus $a$ and $t$ do not change when passing from $C$ to $C'$.
The statement about $\delta$ is Proposition~\ref{prop:dividing},
and the statement about $\ncr$ is Corollary~\ref{cor:crossing}.
\end{proof}

\begin{proof}[Proof of Theorem~\ref{thm:curve-uniqueness}]
Let $C_1$ and $C_2$ be two real irreducible sextics whose only singularities are
$m$ pairs of non-real nodes.
Suppose that $C_1$ and $C_2$ have isotopic real parts,
are of the same dividing type and, if they are dividing,
have the same number of crossing pairs.
Then by Proposition~\ref{prop:corresp}, the invariants $m, a, t, \delta, \ncr$
of their homological types are the same.
Hence by Theorem~\ref{thm:unique} their homological types are isomorphic
and by Proposition~\ref{prop:reflections} the curves $C_1$ and $C_2$ are rigidly isotopic.
\end{proof}

Given an isotopy type of a collection of ovals in $\RP^2$,
a dividing type, and, if the dividing type is I, a number $r$ of crossing pairs,
we can define values $a,t,\delta$ and $r$ by
taking the geometric interpretation in Proposition~\ref{prop:corresp} as a definition.

\begin{lemma}\label{lem:condequiv}
If an isotopy type, a dividing type and a number of crossing pairs satisfy
conditions \textup{(1)--(4)} given in the introduction,
then the corresponding values $a,t,\delta$ and $r$ satisfy the conditions
\ref{item:cond1}--\ref{item:cond5} of Theorem~\ref{thm:conditions}.
\end{lemma}

\begin{proof}
Properties \ref{item:cond1} and \ref{item:cond2} follow from
analogous properties for non-singular sextics (see \cite{Nikulin79}*{Theorem~3.4.3})
using the existence of a non-singular sextic with the given isotopy type and the required dividing type (1).
Using Proposition~\ref{prop:corresp}, it is straightforward to verify that the properties
\ref{item:cond3}, \ref{item:cond4} and \ref{item:cond5} follow from
Harnack's inequality (2) for $\tilde C$, Arnold's congruence (3) and
the complex orientation formula for curves without injective pairs (4),
respectively. 
\end{proof}

\begin{lemma}\label{lem:injective}
Different combinations of isotopy type, dividing type and number of crossing pairs,
satisfying conditions \textup{(1)--(4)}, lead to different invariants $a,t,\delta$ and $r$.
\end{lemma}
\begin{proof}
This follows essentially from the fact that, for non-singular real sextics,
the rigid isotopy type is determined by $a$, $t$ and $\delta$ (see \cite{Nikulin79}).
\end{proof}

\begin{proof}[Proof of Theorem~\ref{thm:curve-existence}]
It is clear that the conditions (1)--(4) of Theorem~\ref{thm:curve-existence} are necessary
for the existence of a sextic with the desired invariants. 
To show that they are sufficient,
suppose we are given an isotopy type of the real part, a dividing type,
a number of pairs of nodes, and (if the dividing type is I) a number of crossing pairs,
subject to the conditions (1)--(4).
By Lemma~\ref{lem:condequiv}, the corresponding values $(a,t,\delta,r)$
satisfy the conditions \ref{item:cond1}--\ref{item:cond5} of Theorem~\ref{thm:conditions}.
Hence by Theorem~\ref{thm:conditions}, there is a homological type $(\KL,\HTp)$
with these invariants.
Since the real period space $\OHp^0$ is not empty
and the period map is surjective (see \S~\ref{sec:periods.real}),
there are sextics with this homological type.
Finally, Lemma~\ref{lem:injective} guarantees that such sextics indeed have the desired
isotopy type, dividing type and number of crossing pairs.
\end{proof}

\end{document}